\documentclass[times,sort&compress,3p]{elsarticle}
\journal{Journal of Multivariate Analysis}
\usepackage{amsmath, amssymb, enumerate, amsthm}

\usepackage[utf8]{inputenc}
\RequirePackage[unicode=true,colorlinks,citecolor=blue,urlcolor=blue]{hyperref}

\usepackage{amsmath,amsbsy,amsthm,amssymb}
\usepackage{multirow,multicol,tabularx,booktabs}
\usepackage{graphicx,placeins,url}
\usepackage{bbm,bm,mathrsfs}
\usepackage[ruled]{algorithm2e}
\usepackage[shortlabels]{enumitem}
\usepackage[dvipsnames]{xcolor}
\usepackage{scalerel}

\usepackage{tikz}
\usetikzlibrary{shapes,arrows,positioning,fit}

\newtheorem{thm}{Theorem}

\newtheorem{prop}[thm]{\bf Proposition}

\newtheorem{defi}{Definition}
\newtheorem{cond}{\bf Condition}
\newtheorem{example}{Example}

\newtheorem{rem}{\textbf{Remark}}
\let\oldproofname=\proofname
\renewcommand{\proofname}{\rm\bf{\oldproofname}}

\newcommand{\pp}{\mbox{\boldmath $p$}}

\def\indistr{\buildrel {d} \over =}

\def\Ac{\mbox{$\mathcal A$}}

\def\Dc{{\mathcal D}}
\def\Ec{\mbox{$\mathcal E$}}
\def\Fc{\mbox{$\mathcal F$}}
\def\Gc{{\mathcal G}}

\def\Ic{\mbox{$\mathcal I$}}

\def\Mc{\mbox{$\mathcal M$}}
\def\Nc{\mbox{$\mathcal N$}}

\def\Tc{\mbox{$\mathcal T$}}
\def\Uc{\mbox{$\mathcal U$}}

\def\Xc{\mbox{$\mathcal X$}}

\def\Gb{{\mathbb G}}

\def\Lb{{\mathbb L}}
\def\Nb{{\mathbb N}}
\def\Gb{{\mathbb G}}
\def\Pb{{\mathbb P}}

\def\Rb{\mbox{$\mathbb R$}}

\def\EE{ {\rm I} \kern-.15em {\rm E} }
\def\PP{ {\rm I} \kern-.15em {\rm P} }

\def\u{ {\bf u}}

\def\x{ {\bf x}}
\def\y{ {\bf y}}
\def\z{ {\bf z}}

\def\A{ {\bf A}}
\def\I{ {\bf I}}

\def\U{ {\bf U}}
\def\V{ {\bf V}}

\def\X{ {\bf X}}
\def\Y{ {\bf Y}}
\def\Z{ {\bf Z}}
\def\mubf{ {\bm \mu}}
\def\Sigmabf{ {\bm \Sigma}}
\def\Omegabf{ {\bm \Omega}}
\def\lambdabf{ {\bm \lambda}}

\def\0{{\mathbf 0} }
\def\1{{\mathbf 1} }
\def\mds{\medskip}

\def\NA{{\texttt{NA}}}

\makeatletter
\algocf@newcmdside@kobe{ForIf@for}{%
    \KwSty{For every\hphantom{f}}%
    \ifArgumentEmpty{#1}\relax{ #1}%
    #2
}
\algocf@newcmdside@kobe{ForIf@if}{%
    \KwSty{if\hphantom{f}}%
    \ifArgumentEmpty{#1}\relax{ #1}%
    #2
}
\algocf@newcmdside@kobe{ForIf@then}{%
    \ifArgumentEmpty{#1}\relax{ #1}%
    \algocf@block{#2}{end}{#3}%
    \par
}
\newcommand\ForIf[3]{%
    \ForIf@for{#1}%
    \ForIf@if{#2}%
    \ForIf@then{#3}%
}
\makeatother

\begin{document}

\begin{frontmatter}

\title{Identifiability and estimation of meta-elliptical copula generators}


\author[alexis]{A.~Derumigny}
\ead{A.F.F.Derumigny@tudelft.nl}
\author[jdf]{J.-D.~Fermanian\texorpdfstring{\corref{cor1}}{}}
\ead{jean-david.fermanian@ensae.fr}
\cortext[cor1]{Corresponding author}
\address[alexis]{Department of Applied Mathematics, Delft University of Technology, Netherlands}
\address[jdf]{Crest-Ensae, 5 av. Henry le Chatelier, 91764 Palaiseau cedex, France}

\begin{abstract}
Meta-elliptical copulas are often proposed to model dependence between the components of a random vector. 
They are specified by a correlation matrix and a map $g$, called density generator. 
While the latter correlation matrix can easily be estimated from pseudo-samples of observations, the density generator is harder to estimate, especially when it does not belong to a parametric family.
We give sufficient conditions to non-parametrically identify this generator. Several nonparametric estimators of $g$ are then proposed, by M-estimation, simulation-based inference, 
or by an iterative procedure available in the \texttt{R} package \texttt{ElliptCopulas}. Some simulations illustrate the relevance of the latter method.
\end{abstract}



\end{frontmatter}

\section{Introduction}

Elliptically contoured distributions are usual semi-parametric extensions of multivariate Gaussian or Student distributions.
They correspond to continuous distributions on $\Rb^d$ whose isodensity curves (with respect to the Lebesgue measure) are ellipsoids: see, e.g., 
\cite{cambanis1981_theory,FangKotzNg,gomez2003_survey,kelker1970_distribution}.
To be specific, let $\X$ be a random vector in $\Rb^d$ whose elliptical distribution is parameterized by a vector $\mubf\in \Rb^d$, a positive definite matrix $\Omegabf=[\Omega_{i,j}]_{1\leq i,j \leq d}$ and a measurable function
$g:\Rb^+ \rightarrow \Rb^+\cup \{+\infty\}$. Its density with respect to the Lebesgue measure in $\Rb^d$ is
    \begin{equation}
    f_{\X}(\x) = {|\Omegabf|}^{-1/2}
    g\left( (\x-\mubf)^\top \, \Omegabf^{-1} \, (\x-\mubf) \right),\;\x \in \Rb^d.
    \label{def_dens_ellip}
    \end{equation}
This distribution is denoted $\Ec_d(\mubf, \Omegabf, g)$ and its cumulative distribution function by $H_{g,\Omegabf,d}$. 
To specify the law of the random vector $\X$, we will write $\X \sim \Ec_d(\mubf, \Omegabf, g)$.
The map $g$ is called ``density generator'', or simply ``generator''.
By integration of~(\ref{def_dens_ellip}), a density generator of an elliptical vector satisfies the constraint
\begin{equation}
    s_d\int_0^{\infty} r^{d-1} g(r^2)\, dr =s_d\int_0^{\infty} t^{d/2-1} g(t)\, dt/2 =1,
    \label{cond_dens_gen}
\end{equation}
where $s_d:=2 \pi^{d/2}/\Gamma(d/2)$ is the surface area of the unit ball in $\Rb^d$, $d\geq 2$ ($s_1=2$).
Conversely, any nonnegative function $g$ that satisfies~(\ref{cond_dens_gen}) can be used as the density generator of an elliptical distribution.
For example, the density generator of a Gaussian distribution is
$ g_{Gauss}(u ) := \exp(-u/2) / (2\pi)^{d/2}$.

\begin{rem}
\textnormal{
It is possible to define elliptical distributions with singular matrices $\Omegabf$: see~\cite{cambanis1981_theory}. 
In such cases, $\lambdabf^\top \X=0$ a.s. for some vector $\lambdabf $ in $\Rb^d$, and the law of $\X$ 
is supported on an affine subspace of $\Rb^d$.
Such ``degenerate'' cases will not be considered in this paper.
}
\end{rem}

A $d$-dimensional copula $C$ is said meta-elliptical (or simply ``elliptical'') if there exists an elliptical distribution in $\Rb^d$ whose copula is $C$. 
Due to the invariance of copulas by location-scale transforms, an elliptical copula depends on a generator and a correlation matrix only.
Indeed, for any
$\X \sim \Ec_d(\0,\Omegabf,g)$, set
$\Y := (X_1/\sqrt{\Omega_{1,1}},\ldots,X_d/\sqrt{\Omega_{d,d}})^\top$. Then, $\Y \sim \Ec_d(\0,\Sigmabf,g)$ for a correlation matrix
$\Sigmabf:=[\Omega_{i,j}/(\Omega_{i,i}\Omega_{j,j})^{1/2}]$. Obviously, the meta-elliptical copulas of $\X$ and $\Y$ are the same.
Thus, for any correlation matrix $\Sigmabf$, denote by $\Mc\Ec_d(\Sigmabf, g)$ the (unique) meta-elliptical copula that corresponds to the elliptical distribution $\Ec_d(\0, \Sigmabf, g)$. More generally, this copula corresponds to the elliptical distributions $\Ec_d(\mu, \Omegabf, g)$ for any $\mu, \Omegabf$, such that $\Sigmabf$ is the correlation matrix associated to $\Omegabf$.

\mds

A trans-elliptical distribution \cite{liu2012_transelliptical, liu2016semiparametric} is a distribution whose copula is meta-elliptical.
Trans-elliptical distributions (\cite{fang2002_metaelliptical}) extend elliptical distributions, by allowing 
the associated margins to be arbitrarily specified. 
For any correlation matrix $\Sigmabf$, we will denote by $\Tc\Ec_d(\Sigmabf, g, F_1, \dots , F_d)$ the trans-elliptical distribution whose copula is $\Mc\Ec_d(\Sigmabf, g)$ and marginal cdfs' are $F_1, \dots, F_d$. 

\mds 

The probabilistic properties of meta-elliptical copulas and their statistical analysis have been studied in several papers in the literature:
their conditional distributions and dependence measures~\cite{fang2002_metaelliptical}, the estimation of the correlation matrix $\Sigma$~\cite{wegkamp2016adaptive}, 
stochastic ordering~\cite{abdous2005dependence},
sampling methods~\cite{wang2013practical}, their tail dependence
function~\cite{kluppelberg2007estimating, kluppelberg2008semi, kostadinov2005_nonparametric}, their use semi-parametric regressions~\cite{zhao2019inference} or some goodness-of-fit tests~\cite{jaser2017simple,jaser2020tests}, their relationships with partial and conditional correlations~\cite{kurowicka2000_elliptical},
etc., are notable contributions.
Thanks to Sklar's theorem, meta-elliptical copulas can be used as key components of many flexible multivariate models, 
far beyond elliptically-distributed random vectors. Moreover, the literature has considered parametric families of generators that include 
the popular Gaussian and/or Student copulas as particular cases, with 
practical applications in hydrology \cite{genest2007_metaelliptical, song2010meta}, risk management \cite{embrechts2002_correlation, frahm2003elliptical}, econometrics \cite{sancetta2009_forecasting}, biology, etc.

\mds

To the best of our knowledge, all these papers assume the generator of a meta-elliptical copula is known, 
possibly up to finite dimensional parameter. Thus, the problem of estimating $g$ itself is bypassed.
For instance,~\cite{genest2007_metaelliptical} proposed a graphical tool to select a ``well-suited'' generator among a finite set of potential generators.
Approximations of meta-elliptical copulas are proposed in \cite{touboul2011goodness} through projection pursuit techniques but the consistency of the proposed algorithm seems to occur only under very restrictive conditions.
Actually, a general nonparametric estimation of the generator $g$ is problematic.
As noticed in Genest et al. \cite{genest2007_metaelliptical} : ``The estimation of g is more complex, considering that it is a functional parameter. Indeed, a rigorous approach to this
problem has yet to be developed. Financial applications to date have simply treated g as fixed; however, several possible choices of g have often been considered to assess the robustness of
the conclusions derived from the model.''
Therefore, until now, no nonparametric consistent estimator of $g$ seems to be available in the literature. This should not be surprising. Indeed, a preliminary point would be to state the
identifiability of $g$ from the knowledge of the underlying copula. This result is far from obvious and is one of the main contributions of our work.

\mds

Let us explain why this is the case. 
As recalled in the appendix, all margins of a distribution $\Ec_d(\0,\Sigmabf,g)$ have the same density $f_g$, where, for every $t \in \Rb$,
\begin{equation}
    f_g(t) = \frac{\pi^{(d-1)/2}}{\Gamma((d-1)/2)}\int_0^{+\infty} g(t^2+s) s^{(d-3)/2} \, ds.
\label{def_density_fg}
\end{equation}
Note that $f_g$ is even.
For notational convenience, we do not write the dependency of $f_g$ on the dimension $d$.
Set its marginal cdfs' $F_g(x)=\int_{-\infty}^x f_g(t)\, dt$ for every real number $x$, and its quantile function $Q_g(u)=\inf \{x ; F_g(x) \geq u\}$, $u\in [0,1]$.
By Sklar's theorem, our meta-elliptical distribution $C\sim \Mc\Ec_d(\Sigmabf, g)$ is given by
\begin{align}
    C(\u) &= H_{g,\Sigmabf,d}\big(Q_g(u_1),\ldots , Q_g(u_d)\big) = {|\Sigmabf|}^{-1/2}\mathlarger{\int_{- \infty}^{Q_g(u_1)}\cdots  \int_{- \infty}^{Q_g(u_d)}}
    g \big( \x^\top\Sigmabf^{-1} \x \big) \, d\x,
    \label{def_cop_meta_ellip}
\end{align}
for every $\u\in (0,1)^d$.
Hence, the associated meta-elliptical copula density with respect to the Lebesgue measure exists and may be defined as
\begin{equation}
    c(\u)
    := \frac{g \big( \vec{Q}_g(\u) \Sigmabf^{-1} \vec{Q}_g(\u)^\top \big)}
    {|\Sigmabf|^{1/2}\prod_{k=1}^d f_g\big(Q_g(u_k)\big)
    },\;\; \text{where }
    \vec{Q}_g(\u) := \big[ Q_g(u_1),\ldots,Q_g(u_d)\big],
    \label{def_cg_mle}
\end{equation}
for every $\u\in [0,1]^d$.
The latter density and cumulative distribution function will be denoted by $c_g$ and $C_g$ respectively, 
when we want to stress its dependence with respect to $g$.

\mds



\mds

Concerning the inference of meta-elliptical copulas, the usual estimator of the matrix $\Sigmabf$ has been known for a long time and is based on empirical Kendall's tau (see below). 
When one observes i.i.d. realizations of elliptically-distributed random vectors, several estimators of the generator $g$ have been proposed in the literature (see our appendix).
This paper is related to the same purpose, but for trans-elliptical distributions, or, equivalently, for meta-elliptical copulas.
In such a case, the inference of $g$ is more difficult than for elliptical laws because copula densities depend on $g$ through a highly nonlinear and complex relationship. Moreover, it is not known whether the mapping $g \mapsto c_g$ is one-to-one.
To the best of our knowledge, this problem has never been tackled in the literature. Authors only rely on ad-hoc chosen generators, or on
parametric families of generators.

\mds

In Section~\ref{identif}, we give sufficient conditions for the identifiability of the generator of a meta-elliptical copula.
Estimation procedures of $g$ are proposed in Section~\ref{inference_section}.
They allow the nonparametric estimation of an assumed trans-elliptical distribution, because $\Sigmabf$ and its margins are easily estimated beside.
Section~\ref{sec:NPMLE} proposes a penalized M-estimation of $g$, when a simulation-based estimator is given in Section~\ref{simul_dens_ellip}.
Section~\ref{sec:iteration_algorithm_estimation} states a numerical iterative procedure to evaluate $g$. 
It will be called MECIP, as ``Meta-Elliptical Copula Iterative Procedure''. Its performances are evaluated in Section~\ref{numerical_results} and it is implemented in the \texttt{R} package \texttt{ElliptCopulas}~\cite{package_ElliptCopulas}. 

\section{Identifiability of meta-elliptical copulas}
\label{identif}

Consider a meta-elliptical copula $C=\Mc\Ec_d(\Sigmabf, g)$ where $\Sigmabf$ is a correlation matrix and $g$
satisfies the usual normalization constraint~(\ref{cond_dens_gen}).
The question is to know whether the latter parameterization is unique.
Strictly speaking, this is the same question as for trans-elliptical distributions when their margins are known.

\mds

For any meta- or trans-elliptical distribution, the correlation matrix $ \Sigmabf$ is identifiable.
Indeed, it is well-known there exists a nice relationship between its components and the underlying Kendall's tau:
$\Sigma_{k,l}=\sin(\pi \tau_{k,l}/2)$, $k\neq l$, where $\tau_{k,l}$ is the Kendall's tau between $X_k$ and $X_l$.
See \cite{wegkamp2016adaptive} and the references therein, for instance.
Since every $\tau_{k,l}$ is uniquely defined by the underlying copula of $\X$, this is still the case for $\Sigmabf$ too.

\mds

\begin{prop}
    If $\U \sim \Mc\Ec_d(\Sigmabf, g)$ and $\U \sim \Mc\Ec_d(\tilde \Sigmabf, \tilde g)$ where $\Sigmabf$ and $\tilde \Sigmabf$ are two correlation matrices, then $\Sigmabf = \tilde \Sigmabf$.
\end{prop}

\textcolor{black}{Note that it is possible to have $g\neq \tilde g$ due to the non-uniqueness of meta-elliptical copula generators (see Proposition~\ref{non_identif_generator} below).}
If $\X$ is a trans-elliptical random vector, the inference of the matrix $\Sigmabf$ can be done independently of the margins.
The estimated matrix $\hat\Sigmabf:=[\hat\Sigma_{k,l}]$ is given by
$\hat\Sigma_{k,l} :=\sin( \pi \hat\tau_{k,l}/2)$, $k\neq l$, and $\hat\Sigma_{k,k}=1$,
introducing empirical Kendall's tau
  \begin{equation*}
    \hat\tau_{k,l} := \frac{2}{n(n-1)} \sum_{1\leq i<j\leq n } \text{sign}( X_{i,k}-X_{j,k}) \times \text{sign}( X_{i,l}-X_{i,l}),\;\; k\neq l.
\end{equation*}
Since it is not guaranteed that $\hat\Sigmabf$ is a correlation matrix, this can be imposed by 
projection techniques (\cite[Section 8.7.2.1]{remillard2013statistical}, e.g.). Hereafter, we require that $\hat\Sigmabf$ is invertible. 

\mds

Having tackled the identifiability and the estimation problem of $\Sigmabf$, the problem is reduced to the following one:
let $g$ and $\bar g$ be two density generators of meta-elliptical copulas on $[0,1]^d$ such that $c_g = c_{\bar g}$, 
with the previous notations. Does it imply that $g=\bar g$ almost everywhere?

\begin{rem}
\textnormal{
There is a one-to-one mapping between generators $g$ of meta-elliptical copulas and the so-called univariate densities $f_g$, as given in~(\ref{def_density_fg}), once the underlying copula density $c$ is known.
Indeed, since $\Sigmabf^{-1}$ is definite positive, its diagonal elements are positive. Then, invoke~(\ref{def_cg_mle}) 
with 
$\vec{Q}_g(\u) := \big[ x,0,\ldots,0\big]$ for some arbitrary $x\in \Rb$.
Since $f_g$ is even, $F_g(0)=1/2$ and $f_g(0) \neq 0$ by~(\ref{def_density_fg}). This yields $g( \gamma x^2)=|\Sigmabf|^{1/2} f_g(0)^{d-1} f_g(x) c\big(F_g(x),1/2,\ldots,1/2\big)$ for every $x$ and some known positive constant $\gamma$.
This means the map $g \mapsto f_g$ is invertible, restricting ourselves to meta-elliptical copula generators.
In other words, since copula densities are nonparametrically identifiable, 
the identifiability problem of $g$ or of $f_g$ are the same. 
But, since $f_g$ can be (nonparametrically) identified only in the case of elliptical distributions, this 
does not prove the identifiability of $g$ for general meta-elliptical/trans-elliptical distributions.
}
\end{rem}

Recall that elliptical distributions $\Ec_d(\mubf, \Omegabf, g)$ are not identifiable in general without any identifiability constraint (see Proposition~\ref{prop:ellipt_distr_ident} in the appendix).
Therefore, most authors impose a condition such as $\text{Cov}(\X)=\Omegabf$ (when $\X$ has finite second moments) or
$\text{Tr}(\Omegabf)=1$ (in the general case).
To deal with meta-elliptical copulas, we are facing similar problems.
Indeed, such distributions are never identifiable without identifiability constraints, as proven in the next proposition.


\mds

\begin{prop}
    Let $\Sigmabf$ be a positive definite correlation matrix and $g$ be a density generator of a meta-elliptical copula on $[0,1]^d$.
    Then, for any $a>0$, $\Mc\Ec_d(\Sigmabf, g) = \Mc\Ec_d(\Sigmabf, g_a)$ by setting $g_a(t) := a^{d/2} g(a \times t)$.
\label{non_identif_generator}
\end{prop}

\begin{proof}[\bf{Proof of Proposition~\ref{non_identif_generator}}]
By~(\ref{def_density_fg}), we easily get $f_{g_a}(t)=\sqrt{a} f_g(\sqrt{a}t)$ for every $t$. Deduce $F_{g_a}(x)=F_g(\sqrt{a}x)$ and 
$Q_{g_a}(u)=Q_g(u)/\sqrt{a}$ for every $t$ and $u$. 
Then, applying (\ref{def_cg_mle}), we check that the copula densities associated to $g$ and $g_a$ respectively are the same. 
\end{proof}

\mds
Therefore, for a given meta-elliptical copula, its generator $g$ has to satisfy at least another constraint in addition to~(\ref{cond_dens_gen}), to be uniquely defined.
We will prove that the generator of any elliptical copula is identifiable under some regularity conditions.
This is one of the main contributions of our work.
Before, we need our density generators to be sufficiently regular so that the associated univariate densities $f_g$ are differentiable. This is guaranteed by the next assumption.
\begin{cond}
\textnormal{
\label{reg_generator}
Set $a_g:=\sup \{t \,|\, t>0, g(t^2)>0\} \in \bar \Rb_+$. 
The map $t\mapsto g(t^2)$ from $\Rb$ to $\Rb^+$ is strictly positive and differentiable on $(-a_g,a_g)$.
Moreover, the map $t \mapsto \int_0^{+\infty} g(t^2+r^2)r^{d-2}\, dr$ is finite and differentiable on $\Rb$.
}
\end{cond}
As a consequence, $\{t \,|\, g(t)>0\}$ (resp. $\{t \,|\, f_g(t)>0\}$) is equal to the interval $(-a_g^2,a_g^2)$ (resp. $(-a_g,a_g)$), possibly including the boundaries.
Thus, we do not allow generators whose supports exhibit ``holes'', such as sums of indicator functions that are related to disjoint subsets.
The forbidden models correspond to meta-elliptical copulas
whose densities are zero in some ``cavities'' that look like ``distorted rings'' (when plotted on $[0,1]^d$).
Such features can easily be identified by plotting nonparametric estimates of $c_g$ as a preliminary stage. 
They are also unlikely to happen in practical applications.

\mds

Denote by $\Gc$ the set of density generators $g$ that satisfy Condition~\ref{reg_generator}, in addition to~(\ref{cond_dens_gen}).
They will be called ``regular density generators''.
Moreover, denote by $ c_{i,j}$ the map from $[0,1]^2$ to $\Rb$ that is equal to the copula density $c$, when all its arguments are equal to $1/2$, except the $i$-th and the $j$-th.

\begin{prop}
Consider a meta-elliptical random vector $\U \sim \Mc\Ec_d(\Sigmabf, g)$, where
the correlation matrix $\Sigmabf $ is not the identity matrix $\I_d$.
Let $i$ and $j$ be two different indices in $\{1,\ldots,d\}$ for which
the $(i,j)$-component of $\Sigmabf^{-1}$ is not zero.
Assume that $c_{i,j}$ has finite first-order partial derivatives on $(0,1)^2$.
Moreover, the maps $u\mapsto \partial_k \ln c_{i,j}(u,1/2)$, $k\in \{1,2\}$, are locally Lipschitz on $(0,1)$.
Consider two couples $(\Sigmabf_k,g_k)$, $g_k \in \Gc$, $k\in \{1,2\}$, that induce the same law of $\U$ and that both satisfy
\begin{equation}
b:=\frac{\pi^{(d-1)/2}}{\Gamma((d-1)/2)}\int_0^{+\infty} g_k(s) s^{(d-3)/2} \, ds,
\label{cond_density_fg}
\end{equation}
for a given positive real number $b$.
Then these couples are essentially unique: $\Sigmabf_1=\Sigmabf_2$; the supports of $g_1$ and $g_2$ are the same interval (except possibly at its boundaries), 
and $g_1=g_2$ on the interior of their common support.
\label{prop_identif_g}
\end{prop}

{\color{black}Therefore, $b$ can be arbitrarily chosen. For any generator $g$, it is always possible to find a version $\tilde g$ such that $\Mc\Ec_d(\Sigmabf, g) = \Mc\Ec_d(\Sigmabf, \tilde g)$ for all matrices $\Sigmabf$, and such that $b(\tilde g) = 1$. 
An explicit procedure to build $\tilde g$ is given in Algorithm~\ref{algo:normalization_g}.}
Note that~(\ref{cond_density_fg}) means $f_{g_1}(0)=f_{g_2}(0)=b$.
By~(\ref{def_density_fg}), $f_{g_1}=f_{g_2}$ and $F_{g_1}=F_{g_2}$ everywhere, implying $C_{g_1}=C_{g_2}$ everywhere by~(\ref{def_cop_meta_ellip}).
Condition~\ref{reg_generator} implies $g(0)>0$, that is equivalent to $c(1/2,1/2,\ldots,1/2)>0$ (see~(\ref{def3_ellip_density}) in the proof).
The latter condition seems to be weak, particularly under a practical perspective: copulas that have no mass at the center of their support 
can be considered as ``pathological''. 

\mds

Therefore, for a given elliptical copula and if $\Sigmabf \neq \I_d$, there ``most often'' exists a unique regular density generator $g$ for which the constraints~(\ref{cond_density_fg})
and
$s_d\int_0^{\infty} t^{d/2-1} g(t)\, dt/2 =1$ are both satisfied.
In other words, two moment-type conditions are sufficient to uniquely identify the generator of an elliptical copula ``most of the time''.
Proposition~\ref{prop_identif_g} shows that, when $d>2$, the generator of a $d$-dimensional elliptical copula is uniquely defined by
these two moment conditions in addit\textcolor{black}{i}on to the knowledge of a single map
$ (u,v) \mapsto c_{i,j}\big(u,v\big)$, $i$ and $j$ being two indices in $\{1,\ldots,d\}$ for which
the $(i,j)$-component of $\Sigmabf^{-1}$ is not zero. In other words, it is not necessary to know the whole copula density $c$ on $[0,1]^d$ to recover
the density generator of an elliptical copula. Only a single bivariate cross-section is sufficient.

\textcolor{black}{
\begin{example}
\textnormal{
\label{ex_Kotz}
Consider the particular case of meta-elliptical copulas whose generator is given by
$g(t)= P(t) \exp(-\lambda t ^{\beta})$ for some positive constants $\lambda, \beta $ and some polynomial $P$ such that $P(0)>0$ and $P(t)\geq 0$ for every $t\geq 0$.
They are linear combinations of the generators associated to the family of symmetric Kotz-type distributions 
(see~\cite{fang2002_metaelliptical}, Example 2.1), 
including Gaussian copulas as particular cases. 
They satisfy Condition~\ref{reg_generator} and then Proposition~\ref{prop_identif_g} applies to them.
}
\end{example}
\begin{example}
\textnormal{
\label{ex_Pearson}
When a meta-elliptical copula generator is 
$g(t)= K_{m,N} \big(1+t/m \big)^{-N}$
for some positive constants $K_{m,N}, m$ and $N>1$, this yields the copulas associated to the family of symmetric bivariate Pearson type VII distributions (see~\cite{fang2002_metaelliptical}, Example 2.2), including Student distributions when $N=m/2+1$.
Check that Condition~\ref{reg_generator} is fulfilled when $N> 1+d/2$.
}
\end{example}
}

\begin{proof}[\bf{Proof of Proposition~\ref{prop_identif_g}}]
First consider the bivariate case $d=2$. Denote by $\rho\neq 0$ the extra-diagonal component of $\Sigmabf$.
With the same notations as above, the copula density $c$ of $\U$ with respect to the Lebesgue measure satisfies
\begin{equation}
 c\big(F_g(x),F_g(y)\big)f_g(x) f_g(y)= \frac{ g\big( (x^2 + y^2 - 2\rho xy)/(1-\rho^2)\big) }{\sqrt{1-\rho^2} },
\label{def2_ellip_density}
\end{equation}
for almost every $(x,y)\in \Rb^2$, by Sklar's theorem.
Here, we clearly see that the maps $g\mapsto F_g$ and $g\mapsto f_g$ are one-to-one for a given copula density $c$:
setting $x=\rho y$, Eq.~(\ref{def2_ellip_density}) yields
\begin{equation} 
g\big( y^2 \big) =  \sqrt{1-\rho^2}c\big(F_g(\rho y),F_g(y)\big)f_g(\rho y) f_g(y) , 
\label{def2_ellip_density_xrhoy}
\end{equation}
for every $y \in \Rb$.
Therefore, the knowledge of $f_g$ (or $F_g$, equivalently) provides a single generator $g$ that induces the given copula $c$.
Now, it is sufficient to prove the identifiability of $f_g$.
By Condition~\ref{reg_generator}, $(-a_g,a_g)$ is the support of $f_g$ (possibly including the boundaries) and $f_g(0)=b$ is positive. 
Setting $y=0$ in~(\ref{def2_ellip_density}), we get
\begin{equation*}
 \sqrt{1-\rho^2} c\big(F_g(x),1/2\big)f_g(x)b= g\big( x^2 /(1-\rho^2)\big) ,
\label{def2_ellip_density_y0}
\end{equation*}
for every $x\in \Rb$.
Since $g$ is non zero and continuous at zero, there exists an open neighborhood of zero $V_0$ for which 
$c\big(F_g(x),1/2\big)>0$ when $x\in V_0$.

\mds

We will restrict ourselves to the couples $(x,y)\in \Rb^2$ such that
the non-negative number $\{x^2 + y^2 - 2\rho xy\}/(1-\rho^2)$ belongs to $[0,a_g^2)$. Denote by $\Xc(g)$ the set of such couples. Note that $\Xc(g)$ contains an open neighborhood of $(0,0)$ and that $f_g(x)f_g(y)>0$ for such couples, due to~(\ref{def2_ellip_density}).

\mds

Condition~\ref{reg_generator} means that $f_g$ is differentiable on $\Rb$. 
Since it is even, $f'(0)=0$.
By differentiating~(\ref{def2_ellip_density}) with respect to $x$, we get
\begin{equation}
\partial_1 c\big(F_g(x),F_g(y)\big) f_g^2(x)f_g(y) + c\big(F_g(x),F_g(y)\big) f'_g(x)f_g(y)=
g'\Big( \frac{x^2 + y^2 - 2\rho xy}{1-\rho^2} \Big) \frac{2(x-\rho y)}{(1-\rho^2)^{3/2}},
\label{deriv_fg_at_zero}
\end{equation}
for any $(x,y)\in \Xc(g)$.
Set $y=x/\rho$, cancelling the right-hand side of~(\ref{deriv_fg_at_zero}).
This yields
$$ \partial_1 \ln c\big(F_g(x),F_g(x/\rho)\big) f_g^2(x) + f'_g(x)=0,$$
for every $x\in (-a_g,a_g)$. As a consequence, $f'_g$ is continuous on the latter interval.

\mds
By independently differentiating~(\ref{def2_ellip_density}) with respect to $y$ and comparing with~(\ref{deriv_fg_at_zero}), we obtain
\begin{eqnarray}
\lefteqn{(y-\rho x) \big\{  \partial_1 c\big(F_g(x),F_g(y)\big) f_g^2(x)f_g(y) + c\big(F_g(x),F_g(y)\big) f'_g(x)f_g(y)      \big\}  \nonumber }\\
&=&
(x-\rho y) \big\{  \partial_2 c\big(F_g(x),F_g(y)\big) f_g(x)f^2_g(y) + c\big(F_g(x),F_g(y)\big) f_g(x)f'_g(y) \big\},
\label{ODE_identif}
\end{eqnarray}
for every $(x,y)\in \Xc(g)$.
Consider the particular value $y=0$, for which $f_g(y)=b >0$ and $f'_g(y)=0$.
By symmetry, $F_g(x) = 1 - F_g(-x)$ for every real number $x$ and $F_g(0)=1/2$.
Then,~(\ref{ODE_identif}) can be rewritten as follows: 
\begin{equation}
\rho  \big\{  \partial_1 \ln c\big(F_g(x),1/2\big) f_g^2(x) +  f'_g(x)      \big\}
=  - \partial_2 \ln c\big(F_g(x),1/2\big) f_g(x)b,
\label{ODE_z_ini}
\end{equation}
for every $x$ in a sufficiently small neighborhood of zero such that $c(F_g(x),1/2)>0$ (such as $V_0$ above, for instance).
Thus, we have obtained an ordinary differential equation, whose solution $z:=F_g$ would be a function of $x$ when $x$ belongs to a neighborhood of zero.
The latter differential equation can be rewritten as
\begin{equation}
z'' +  \partial_1 \ln c\big(z,1/2\big) (z')^2 +
\partial_2 \ln c\big(z,1/2\big) \frac{z' b}{\rho}  =0 .
\label{ODE_z}
\end{equation}
Setting the bivariate map $\vec z = [z,z']$, we are facing the usual Cauchy problem: find $\vec z$, a function of $x$, such that
$ d\vec z= H(\vec z)\, dx$ and that satisfies $\vec z(0)=[1/2,b]$.
Here, the latter map $H:[0,1]\times \Rb \mapsto \Rb^2$ is 
$$ H(z_1,z_2):=\Big[ z_2 ;  -  \partial_1 \ln c\big(z_1,1/2\big) z_2^2 -
\partial_2 \ln c\big(z_1,1/2\big) \frac{z_2 b}{\rho} \Big].$$
By assumption, this map $H$ is Lipschitz on any subset $[\alpha,\beta]\times [b-\epsilon,b+\epsilon]$, when $0<\alpha < \beta <1$ and $\epsilon>0$.
In particular, this is the case when $\alpha < 1/2 < \beta$.
By the Cauchy-Lipschitz Theorem, we deduce there exists a unique solution $\vec z$ in an open neighborhood of $x=0$. Note that this solution satisfies~(\ref{cond_density_fg}) by construction.

\mds

Therefore, consider a (global) solution $\vec z=[F_g,f_g]$ of~(\ref{ODE_z}) on some maximum interval $S_g$ on the real line that contains zero. 
We can impose the latter solution is associated to a regular generator that satisfies~(\ref{cond_density_fg}).
Now, assume there are two different regular generators $g$ and $\bar g$ that induce the same copula.
Set $z=[F_g,f_g]$, $\bar z=[F_{\bar g},f_{\bar g}]$ and $\Ac:=\{t \in S_g\cap S_{\bar g}; z(t)=\bar z(t)\}$. We have proved that $\Ac$ contains an open ball around zero.
Define $t^*:= \sup \{ t ; t \in \Ac\}$ and assume that $t^*$ is finite.
Assume $t^*< \min( a_g,a_{\bar g})$.
By the continuity of the considered cdfs' and densities (Condition~\ref{reg_generator}), $z(t^*)=\bar z(t^*)$ is a known value $v^*$.
Then, we can apply again the Cauchy-Lipschitz Theorem to the differential equation~(\ref{ODE_z}), with the condition $z(t^*)=v^*$. This yields a unique solution of~(\ref{ODE_z}) in an open neighborhood of
$t^*$. As a consequence, $z(t^*+\epsilon)=\bar z(t^*+\epsilon)$ for some $\epsilon>0$. This contradicts the definition of $t^*$.
Assuming w.l.o.g. $a_g \leq a_{\bar g}$, this implies $t^* \geq a_g$.
In other words, $f_g$ (resp. $F_g$) and $f_{\bar g}$ (resp. $F_{\bar g}$) coincide on $(-a_g,a_g)$, implying $g=\bar g$ on the $(-a^2_g,a^2_g)$  (recall~(\ref{def2_ellip_density_xrhoy})). 
To satisfy~(\ref{cond_dens_gen}) with $\bar g$, this requires $\tilde a_g=\tilde a_{\bar g}$, i.e., $a_g=a_{\bar g}$. 
Thus, $g=\bar g$ on the interior of their common support. This proves the result when $d=2$.

\mds

Second, for an arbitrary dimension $d>2$, there exists a non-zero extra-diagonal element in $\Sigmabf^{-1}$ by assumption.
W.l.o.g., assume it is corresponding to the couple of indices $(1,2)$.
Let us fix the other arguments of the copula density $c$ at the value $1/2$, i.e., we focus on the points $(x,y,0,\ldots,0)$ in $\Rb^d$.
This implies there exist two non zero real numbers $(\theta,\gamma)$ such that
\begin{equation}
 c\bigg(F_g(x),F_g(y),\frac{1}{2},\ldots,\frac{1}{2}\bigg)=
 \frac{ g\big(\gamma x^2 + \gamma y^2 +2\theta  xy \big) }{ |\Sigmabf|^{1/2} f_g(x) f_g(y)b^{d-2}},
\label{def3_ellip_density}
\end{equation}
 for every $(x,y)\in \Rb^2$.
By differentiation with respect to $x$ and $y$ respectively, we get an ordinary differential equation that is strictly similar to~(\ref{ODE_z}), apart from different non zero constants.
By exactly the same arguments as in the bivariate case, we can prove there exists a unique global solution on the real line of this differential equation.
As a consequence, $g$ is uniquely defined by $c$ (except at the boundaries of its support), proving the result.
\end{proof}

Unfortunately, the limiting case $\Sigmabf=\I_d$ cannot be managed similarly by considering differential equations and some initial conditions at the particular point $x=0$.
This is due to the nullity of both side\textcolor{black}{s} of~(\ref{ODE_z_ini}): 
differentiate~(\ref{def2_ellip_density}) with respect to $y$, set $y=0$, and deduce that $\partial_2 \ln c\big( F_g(x),1/2\big)=0$ for every $x$ in a neighborhood of zero.
Nonetheless, we can provide partial answers to this problem by imposing some conditions at $+\infty$.
This requires restricting ourselves to a smaller class of generators.
To this end, we introduce a measurable map $\psi: \Rb^+ \rightarrow \Rb$.

\begin{cond}
\textnormal{
\label{reg_generator_bis}
The density generator $g:\Rb^+ \rightarrow \Rb^+$ belongs to $\Gc$, with $a_g=+\infty$.
Moreover, for every $x\in \Rb$, the map $y\mapsto\big( g'/g \big)(x^2+y^2)$ has a finite limit $\psi(x^2)$, when $y\rightarrow +\infty$.
}
\end{cond}
Denote by $\tilde\Gc_\psi$ the set of density generators $g$ that satisfy Condition~\ref{reg_generator_bis}, in addition to~(\ref{cond_dens_gen}).
Note that such $g$ are assumed to be strictly positive on $\Rb^+$.
For a lot of reasonable generators, we can hope $\psi(x)=0$.
This is the case for all generators that are sums of maps of the form $P(x)\exp(-\lambda x^\beta)$ for some polynomials $P$ and some constants $\lambda >0$ and $\beta \in (0,1)$.
When $\beta=1$, we get a family of ``Gaussian-type'' generators, for which $\psi(x)=-\lambda$ for every $x$. But Condition~\ref{reg_generator_bis} is not fulfilled {\color{black}for} such generators when $\beta>1$.

\begin{prop}
Consider a meta-elliptical random vector $\U \sim \Mc\Ec_d(\I_d, g)$.
Assume that, for every $u\in (0,1)$, the map $v\mapsto \partial_1 \ln c(u,v)$ exists and has a finite limit $\chi(u)$ when $v\rightarrow 1$, $v<1$.
Moreover, $\chi$ is locally Lipschitz on $(0,1)$.
For a given map $\psi$, consider two couples $(\Sigmabf_k,g_k)$, $g_k \in \tilde\Gc_\psi$, $k\in \{1,2\}$, that induce the same law of $\U$ and that satisfy~(\ref{cond_density_fg}).
Then these couples are essentially unique: $\Sigmabf_1=\Sigmabf_2$ and $g_1=g_2$ on their common support $\Rb^+$.
\label{identif_cop_ellip_Id}
\end{prop}

\begin{proof}[\bf{Proof of Proposition~\ref{identif_cop_ellip_Id}}]
Let us assume first that $d=2$.
By definition, the copula density $c$ satisfies
\begin{equation}
 c\big(F_g(x),F_g(y)\big) f_g(x) f_g(y)=g(x^2+y^2),
\label{rel_cop_density_Id}
\end{equation}
for every $(x,y)\in \Rb^2$.
Since the support of $g$ is $\Rb_+$, the support of $f_g$ is the whole real line.
Differentiating the latter equation with respect to $x$ and dividing the new one by both members of~(\ref{rel_cop_density_Id}), we deduce
$$ \partial_1 \ln c\big(F_g(x),F_g(y)\big) f_g(x) + \frac{f_g'}{f_g}(x) = 2x \Big( \frac{g'}{g} \Big)(x^2+ y^2), $$
for every real numbers $x$ and $y$.
Now, let us make $y$ tend to $+\infty$.
From Condition~\ref{reg_generator_bis}, we
deduce, for every $x\in \Rb$,
$$\chi\big( F_g(x)\big) f_g(x) + \frac{f_g'}{f_g}(x) = 2x \psi(x). $$
The latter equation is a second-order differential equation with respect to the unknown function $F_g=:z$, i.e.,
$$\chi\big( z\big) z' + \frac{z''}{z'}(x) = 2x \psi(x), \; \text{or}  \; z'' = 2x z'\psi(x)- \chi\big( z\big) (z')^2 .$$
As in the proof of Proposition~\ref{identif_cop_ellip_notId}, consider the initial conditions $z(0)=1/2$ and $z'(0)=b$.
By a similar reasoning (Cauchy-Lipschitz theorem), we prove the result when $d=2$.

\mds

When $d>2$, we consider the map $(u,v)\mapsto c(u,v,1/2,\ldots,1/2)$ that satisfies
\begin{equation*}
 c\big(F_g(x),F_g(y),\frac{1}{2},\ldots,\frac{1}{2}\big) f_g(x) f_g(y) b^{d-2}=g(x^2+y^2),
 \label{rel_cop_density_Id_dim_d}
 \end{equation*}
for every $(x,y)\in \Rb^2$. The same reasoning as for the case $d=2$ proves the result.
\end{proof}

\textcolor{black}{
It can be checked that the meta-elliptical copulas of Example~\ref{ex_Kotz} satisfy Condition~\ref{reg_generator_bis}
when $\beta\leq 1 $ and Proposition~\ref{identif_cop_ellip_Id} applies to them.  
This is still the case for the copulas of Example~\ref{ex_Pearson}, for any value of $(m,N)$, $N>1+d/2$.
}



\mds 

Finally, as shown in~\cite{abdous2005dependence}, Proposition 1.1, the 
identifiability of $g$ may be obtained in the particular case of Gaussian copulas.
Let us extend the latter result in dimension $d\geq 2$.
\begin{prop}
    Let $\U \sim \Mc\Ec_d(\Sigmabf, g)$ and $\U \sim \Mc\Ec_d(\I_d, g_{Gauss})$ where $\Sigmabf$ is a correlation matrix.
    Then $\Sigmabf = \I_d$ and $g = g_{Gauss}$ a.s.
\label{identif_cop_ellip_notId}
\end{prop}

\begin{proof}[\bf{Proof of Proposition~\ref{identif_cop_ellip_notId}}]
The first part of the proposition is the result of the identifiability of the correlation matrix $\Sigmabf$.
Since $\U \sim \Mc\Ec_d(\I_d, g_{Gauss})$, check that $U_1, \dots, U_d$ are mutually independent. Thus,
$Q_g(U_1), \dots, Q_g(U_d)$ are independent variables and their joint law is an elliptical distribution $\Ec_d(\0, \I_d, g)$.
Lemma 5 in~\cite{kelker1970_distribution} implies that $\big(Q_g(U_1), \dots, Q_g(U_d)\big) \sim \Nc(0, \I_d)$, or $\Ec_d(\0, \I_d, g_{Gauss})$ equivalently.
Using Proposition~\ref{prop:ellipt_distr_ident} in the appendix, this yields $g = g_{Gauss}$.
\end{proof}

\section{Inference of density generators of meta-elliptical copulas}
\label{inference_section}

In this section, we define three inference strategies to evaluate $g$, since we now know that such generators are nonparametrically identifiable under some regularity conditions and two moment-type constraints.

\mds

Let $\X$ be a random vector whose distribution is trans-elliptical $\Tc\Ec_d(\Sigmabf, g, F_1, \dots , F_d)$ for a correlation matrix $\Sigmabf$. 
Let $(\X_1,\ldots,\X_n)$ be 
an i.i.d. sample of realizations of $\X$. 
As a particular case, its law could be elliptical $\Ec_d(\0,\Sigmabf, g)$ when all its margins $F_k$ are equal to $F_g$.
Moreover, if its margins are uniformly distributed on $[0,1]$, then the law of $\X$ is given by a meta-elliptical copula $\Mc\Ec_d(\Sigmabf, g)$.
We assume there exists a single generator $g$ such that Condition~(\ref{cond_dens_gen}) and Condition~(\ref{cond_density_fg}) are fulfilled, for some given constant $b>0$, i.e.,
\begin{align}
    s_d\int_0^{+\infty} t^{d/2-1} g(t)\, dt=2,\;\;
    s_{d-1}\int_0^{+\infty} t^{d/2-3/2} g(t)\, dt=2b.
    \label{cond:ell_cop_identifiable}
\end{align}
As proven in Section~\ref{identif}, this is in particular the case when the conditions of Proposition~\ref{prop_identif_g} (when $\Sigmabf\neq \I_d$)
or Proposition~\ref{identif_cop_ellip_Id} (when $\Sigmabf=\I_d$) are satisfied.
Any candidate for the underlying density generator may be normalized to satisfy the two latter conditions, 
for instance through a transform $t\mapsto \alpha g(\beta t)$ with two conveniently chosen
positive constants $\alpha$ and $\beta$.
This leads to the ``normalizing'' Algorithm~\ref{algo:normalization_g}.
In practical terms, the choice of $b$ does not really matter. We simply advise to set $b=1$ by default, our choice hereafter.

\begin{algorithm}[htbp]
\SetAlgoLined
    \vspace{0.1cm}
    \KwIn{An estimate $\hat g$ of the generator for a meta-elliptical copula of dimension $d \geq 2$.}
    Compute $\Ic_1 = \int_0^{+ \infty} t^{ d/2-1} \hat g(t) dt$
    and $\Ic_2 = \int_0^{+ \infty} t^{ d/2-3/2} \hat g(t) dt$ \;
    Set $\beta = \big(b \,s_{d} \, \Ic_1/  (s_{d-1} \,\Ic_2) \big)^2$
    and $\alpha = 2 \beta^{d/2} / (s_d \, \Ic_1)$ \;
    Calculate $\tilde g := \{ t \mapsto \alpha \times \hat g(\beta \times t) \}$ \;

    \KwOut{A modified version $\tilde g$ satisfying the normalization and identification constraints (\textcolor{black}{Eq.~(\ref{cond_dens_gen})} and~(\ref{cond_density_fg})).}
\caption{Normalization of a meta-elliptical copula generator}
\label{algo:normalization_g}
\end{algorithm}

As usual, the marginal distributions $F_j$, $j\in \{1,\ldots,d\}$ of $\X\sim \Tc\Ec_d(\Sigmabf, g, F_1, \dots , F_d)$ will be consistently estimated by their empirical counterparts $\hat F_j$: for every $x\in \Rb$ and $j$, $\hat F_j(x)
:= n^{-1}\sum_{i=1}^n \1(X_{i,j} \leq x)$,
where $\X_i:=(X_{i,1},\ldots,X_{i,d})$.
As announced, the goal is now to propose nonparametric estimators of the generator $g$, assuming that $g$ is identifiable. 
To the best of our knowledge, this paper is the first to propose solutions to this problem in a well-suited rigorous theoretical framework.


\mds
We will use the following notations:
\begin{itemize}
\item $U_k := F_k(X_k)$ and
$\hat U_{i,k} := \hat F_k(X_{i,k})$
for $k\in \{ 1, \dots, d\}$, $i\in \{1, \dots, n\}$. Set $\U:=(U_1,\ldots,U_d)$ and
$\hat \U_i:=(\hat U_{i,1},\ldots,\hat U_{i,d})$, $i\in \{1,\ldots,n\}$.
\item the sample of (unobservable) realizations of $\U$ is $\Uc:=(\U_1,\ldots,\U_n)$; the sample of pseudo-observations $\hat\U$ is
$\widehat\Uc := (\hat\U_1,\ldots,\hat\U_n)$.
\end{itemize}
Since $\X\sim \Tc\Ec(\Sigmabf,g,F_1,\ldots,F_d)$, note that the law of $\U \sim \Mc\Ec(\Sigmabf,g)$ is the meta-elliptical copula $C$ of $\X$.

\subsection{Penalized M-estimation}
\label{sec:NPMLE}

Without any particular parametric assumption, regular generators $g$ are living in the infinite dimensional functional space $\Gc$.
Its subset of generators that satisfy the two identifiability constraints~(\ref{cond_dens_gen}) and~(\ref{cond_density_fg}) will be denoted by $\Ic$.
In practical terms, we could
approximate $\Gc$ by finite dimensional parametric families $\Gc_m=\{ g_{\theta}; \theta \in \Theta_m\}$, where the dimension of $\Theta_m$ is denoted by $p_m$. Most of the time, $p_m\rightarrow \infty$ with $m$, and the family
$(\Gc_m)$ is increasing: $\Gc_m \subset \Gc_{m+1}$, even if this requirement is not mandatory.
Moreover, it is usual that $m=m_n$ and $m_n$ tends {\color{black}to infinity} with $n$, as in the ``method of sieves'' for inference purpose (see the survey~\cite{chen2007large}, e.g.). Nonetheless, we do not impose the latter constraint again.
Ideally, $\cup_m\Gc_m$ is dense in $\Gc$ for a convenient norm (typically, in a $L^r$ space). A less demanding requirement would be to assume 
$g\in\overline{ \cup_m \Gc_m}$, a condition that is sufficient for our purpose.
Therefore, a general estimator of $\theta$ would be
\begin{equation}
 \hat \theta_{n,m} :=\arg\min_{\theta \in \Theta_m} \Gb_n(\theta,\widehat\Uc) + \pp_n(\lambda_n,\theta),
 \label{Mestimator_theta}
 \end{equation}
for some empirical loss function $\Gb_n$, some penalty $\pp_n(\cdot,\cdot)$ and some tuning parameter $\lambda_n$.

\mds

Typically, the loss function is an average of the type
\begin{equation*}
  \Gb_n \big(\theta,(\u_1,\ldots,\u_n)\big) = \frac{1}{n}\sum_{i=1}^n \ell_n(\theta,\u_i),
\label{def_loss_map}
\end{equation*}
for some map $\ell_n$ from $\Theta_n\times \Rb^d$ to $\Rb$.
For instance, for the penalized canonical maximum likelihood method, set $  \Gb_n(\theta,\widehat\Uc) = -\sum_{i=1}^n \ln c_{g_\theta}(\hat \U_i)/n$.
Many other examples of loss functions could be proposed, based on $L^r$-type distances between cdfs', densities or even characteristic functions.

\mds

Concerning the choice of $\Gc_m$, an omnibus strategy could be to rely on Bernstein approximations: 
for any $a>0$, define the family of polynomials
$$ \Gc_{m,a} := \big\{g \in \Ic: g(x)= \sum_{k=0}^{m} b_k (a+x)^k (a-x)^{m-k}  \, \1(x\in [0,a]),\;\; b_k\in \Rb^+ \;\; \forall k\big\}.   $$
Any continuous map $g$ can be uniformly approximated on $[0,a]$ by Bernstein polynomials that are members of $\Gc_{m,a}$, for $m$ sufficiently large (see~\cite{levasseur1984probabilistic}, e.g.). 
If $g(t) \rightarrow 0$ when $t$ tends {\color{black}to infinity}, then, for every $\epsilon>0$, there exists $a(\epsilon)>0$, an integer $m(\epsilon)$ and a map $g_{\epsilon}$ in $\Gc_{m(\epsilon),a(\epsilon)}$ such that $\| g -g_{\epsilon} \|_{\infty}<\epsilon$. If $g$ is compactly supported, simply set $a(\epsilon)$ as the upper bound of $g$'s support.
If $g\in L^r$, $r>0$ and is continuous, then there exists a similar approximation in $L^r$. 
Therefore, the set $\Gc_m$ in~(\ref{Mestimator_theta}) may be chosen as $\Gc_{m(\epsilon),a(\epsilon)}$ for a given $\epsilon \ll 1$, 
obtained with prior knowledge about the true underlying density generator.

\mds

Alternatively, if $g\in L^2(\Rb)$, introduce an orthonormal basis $(h_n)_{n\geq 0}$ of the latter Hilbert space, say the Hermite functions. 
Then, $g$ can be decomposed as $g=\sum_{k\geq 0}< g,h_k> h_k$ and $\Gc_m$ could be defined as 
$$ \Gc_m:=\big\{ h^+=\max(h,0) \,:\, h(x)=\sum_{k=0}^m b_k h_k(x), \;\; b_k \in \Rb\;\; \forall k       \big\}.$$
Note that every latter map $h^+$ is not differentiable at a finite number of points, the roots of $h$. 
Thus, the members of $\Gc_m$ as defined above do not satisfy Condition~\ref{reg_generator}. This can be seen as a theoretical 
drawback and this can be removed by conveniently smoothing the functions of $\Gc_m$. For instance, replace all latter maps $h^+$ by a smoothed approximation 
$x\mapsto h^+ \ast \phi_N(x)$, $\phi_N(t):=N\exp(-\pi N^2 t^2)$, and for some $N\gg1$ and any $t$.

\mds
Another alternative family of generators supported on $\Rb^+$ could be
\begin{eqnarray*}
\lefteqn{
\Gc_m = \Big\{g \in \Ic: g(x)= \sum_{\ell=1}^{m} Q^2_{\ell}(x) \, \exp\big(- (x- \mu_\ell)^{\alpha_\ell}/\sigma_\ell \big), \text{ for some polynomials } Q_\ell,  }\\
&&\hspace{3cm} \text{deg}(Q_\ell) \leq q_m,\;\text { and constants } (\mu_\ell,\alpha_\ell,\sigma_\ell)\in \Rb\times \Rb_+^* \times \Rb_+, \ell\in \{1,\ldots,m\} \,\Big\}, \hspace{8cm}
\end{eqnarray*}
where $(q_m)$ denotes a sequence of integers that tends {\color{black}to infinity} with $m$. The dimension $p_m$ of $\Gc_m$ (or $\Theta_m$, similarly) is then $p_m=m(q_m+4)$. We do not know whether $\cup_m \Gc_m$ is dense in $L^r$ for any $r>0$ and well-chosen sequences $(q_m)$. 
Nonetheless, we conjecture that most ``well-behaved'' density generators can be accurately approximated in some $L^r$ spaces 
by some elements of $\Gc_m$, at least when $m$ and $q_m$ are sufficiently large.

\mds

If we had observed true realizations of $\U$, i.e., if $\Uc$ replaces $\widehat\Uc$ in~(\ref{Mestimator_theta}), then
one could apply some well established theory of penalized estimators: see Fan and Li~\cite{fan2001variable}, Fan and Peng~\cite{fan2004nonconcave} (asymptotic properties), Loh~\cite{loh2017statistical} (finite distance properties), among others.
When the loss $\Gb_n$ is the empirical likelihood and there is no penalty, $\hat\theta_{n,m}$ is called the Canonical Maximum Likelihood estimator of $\theta_0$ (\cite{genest1995semiparametric,
shih1995inferences}), assuming the true density $g=g_{\theta_0}$ belongs to $\Gc_m$. Tsukahara~\cite{tsukahara2005semiparametric, tsukahara2011erratum} has developed the corresponding theory in the wider framework of rank-based estimators.
When the parameter dimension is fixed ($p_m$ is a constant), the limiting law of $\hat\theta_{n,m}$ can be deduced as 
a consequence of the weak convergence of an empirical copula process, here the empirical
process associated to $\widehat\Uc$: see~\cite{berghaus2017weak, fermanian2004weak, ghoudi1998emprical, ghoudi2004empirical}.
Nonetheless, to the best of our knowledge, the single existing general result that is able to simultaneously manage pseudo-observations and penalizations is~\cite{poignard2021finite}. 
The latter paper extends~\cite{loh2017statistical} to state finite distance boundaries for some norms of the difference between the estimated parameter and the true one.
\textcolor{black}{In this paper, we slightly extend their
results to deal with~(\ref{Mestimator_theta}), i.e., with a sequence of parameter spaces $(\Theta_m)$.
This yields the finite distance properties of $\hat \theta_{n,m}$: see~\ref{prop_thetahat_finite_dist}.
}

\subsection{Simulation-based inference}
\label{simul_dens_ellip}

The previous inference strategy is mainly of theoretical purpose and would impose difficult numerical challenges. 
In particular, the log-likelihood criteria require the evaluation of copula densities through $g$, $f_g$ and $F_g$.  
Unfortunately, $F_g$ is a complex map that is not known analytically in general.
Here, we propose a way of avoiding the numerical calculations of $F_g$ , $Q_g$ or $f_g$, 
by using the fact that it is very easy to simulate elliptical random vectors.

\mds

To be specific, consider a trans-elliptical random vector $\X\sim \Tc\Ec_d(\Sigmabf, g, F_1, \dots , F_d)$. Its copula $C$ is then meta-elliptical $\Mc\Ec_d(\Sigmabf,g)$.
By definition and with our notations, $C$ must satisfy
$$ C\big(F_{g}(x_1),\ldots ,F_{g}(x_d)\big)= H_{g,\Sigmabf,d}(x_1,\ldots,x_d),$$
for every $\x:=(x_1,\ldots,x_d)^\top$.
Note that all margins are the same because $\Sigmabf$ is a correlation matrix.
The goal is still to estimate $g$, with a sufficient amount of flexibility. 
To fix the idea, assume that $g$ belongs to $\overline{ \cup_m \Gc_m}$, with the same notations as in Section~\ref{sec:NPMLE}.
\textcolor{black}{Therefore, as in Section~\ref{sec:NPMLE}, the idea would be to approximate the true generator $g$ by some map that belongs to 
a parametric family $\Gc_m$, for some ``large'' $m$.}

\mds

First, let us estimate the copula $C$ non-parametrically, for instance by the empirical copula $C_n$ based on the sample $(\X_i)_{i\in\{1,\ldots,n\}}$.

\mds

Second, assume for the moment that $(g,\Sigmabf)$ is known. 
For any arbitrarily large integer $N$, let us draw a $N$ sample $(\Y_1,\ldots,\Y_N)$ of independent realizations of $\Y \sim H_{g,\Sigmabf,d}$.
This is easy thanks to the polar decomposition of elliptical vectors: $\Y \stackrel{law}{=} R \A^\top \V,$
where $\A^\top \A=\Sigmabf$, $\V$ is uniformly distributed on the unit ball in $\Rb^d$, and $R$ has a density that is a simple function of $g$.
All margins of $\Y$ have the same distributions $F_g$, and denote by $\hat F_g$ an empirical counterpart: for every $y\in \Rb$,
\begin{equation*}
    \hat F_g(y) := \frac{1}{dN}
    \sum_{l=1}^N \sum_{k=1}^d \1(Y_{k,l}\leq y).
\end{equation*}
Moreover, denote by $\hat H_{g,\Sigmabf}$ the joint empirical cdf of $\Y$, i.e., $ \hat H_{g,\Sigmabf}(\y)=N^{-1}\sum_{l=1}^N \1(\Y_{l}\leq \y).$
Then, we expect we approximately satisfy the relationship
$$ C_n\circ \vec F_g(\x):=C_n\big(\hat F_g(x_1),\dots ,\hat F_g(x_d)\big) \simeq \hat H_{g,\Sigmabf}(x_1,\dots,x_d),\; \x\in \Rb^d.$$

\mds
Obviously, since we do not know $g$ nor $\Sigmabf$, we cannot draw the latter sample $(\Y_i)_{i\in \{1,\ldots,N\}}$ strictly speaking.
Then, we will replace $\Sigmabf$ by a consistent estimator $\hat\Sigmabf$. Moreover, for any current parameter value $\theta$, we can generate the $N$ sample 
$(\Y_{\theta,1},\ldots, \Y_{\theta,N})$, $\Y_{\theta,l} \sim \Ec_d(\0,\hat\Sigmabf, g_\theta )$. 
The associated empirical
marginal and joint cdfs' are denoted
by $\hat F_{g_\theta}$ and $\hat H_{g_\theta,\hat \Sigmabf}$.
Therefore, for a fixed $m$, an approximation of $g$ will be given by $g_{\hat \theta_m}$, where
$$ \hat\theta_m := \arg\min_{\theta \in \Theta_m}
 \Dc\big(  C_n\circ \vec F_{g_\theta} ; \hat H_{g_\theta,\hat \Sigmabf} \big),$$
for some discrepancy $\Dc$ between cdfs' on $\Rb^d$  and some parameter set $\Theta_m$ in $\Rb^{p_m}$.
\textcolor{black}{Note that $\hat\theta_m$ implicitly depends on $n$ and $N$.}
For instance, consider
$$ \Dc\big(  C_n\circ \vec F_{g_\theta} ; \hat H_{g_\theta,\hat \Sigmabf} \big) := 
\int \Big\{  C_n\big(\hat F_{g_\theta}(x_1), \dots, \hat F_{g_\theta}(x_d)\big) - \hat H_{g_\theta,\hat \Sigmabf}(\x) \Big\}^2 w(\x) \, d\x,$$
for some weight function $w$.
To avoid the calculation of $d$-dimensional integrals, it is possible to choose 
some "chi-squared" type discrepancies instead, as
$$ 
 \Dc_{\chi}\big(  C_n\circ \vec F_{g_\theta} ; \hat H_{g_\theta,\hat \Sigmabf} \big) :=
\sum_{l=1}^{L}
\Big| \int_{B_l} dC_n\big(\hat F_{g_\theta}(x_1), \dots, \hat F_{g_\theta}(x_d)\big) - \hat H_{g_\theta,\hat\Sigmabf}(d\x) \Big| ,$$
for some partition $(B_1, \dots, B_L)$ of $\Rb^d$. 
\textcolor{black}{
Alternatively, we could replace the measure $w(\x)\,d\x$ with the empirical law of our observations, yielding
$$ \Dc_{emp}\big(  C_n\circ \vec F_{g_\theta} ; \hat H_{g_\theta,\hat \Sigmabf} \big) := 
\frac{1}{N}\sum_{i=1}^N \Big\{  C_n\big(\hat F_{g_\theta}(X_{1,i}), \dots, \hat F_{g_\theta}(X_{d,i})\big) - 
\hat H_{g_\theta,\hat \Sigmabf}(\X_i) \Big\}^2 ,$$
and this modification avoids suffering from the curse of dimensionality. Unfortunately, whatever the chosen criterion, stating the limiting law of $\hat\theta_m$ when 
$N$ and $n$ tend to infinity (or even the limiting law of $g_{\hat\theta_m}-g$ when 
$m$, $N$ and $n$ tend to infinity) seems to be a particularly complex task that lies beyond the scope of this paper.}

\mds

\subsection{An iterative algorithm: MECIP, or ``Meta-Elliptical Copula Iterative Procedure''}
\label{sec:iteration_algorithm_estimation}

In this section, we propose a numerical recursive procedure called MECIP that allows the estimation of $g$.

\mds

The appendix refers to several estimation procedure{\color{black}s} of the density generator of an elliptical distribution.
We select one of them, that will be called $A$.  
Formally, $A$ is the operator that maps an i.i.d. dataset $(\Z_1, \dots, \Z_n) \in \Rb^{d \times n}$ generated from an elliptical distribution $\Ec_d(\mubf, \Sigmabf, g)^{\otimes n}$ to a map $\hat g := A(\Z_1, \dots, \Z_n)$:
\begin{equation}
\begin{matrix}
A: & \Rb^{d \times n}    & \longrightarrow & \Fc \\
 & (\Z_1, \dots, \Z_n) & \mapsto         & \hat g,
\end{matrix}
\end{equation}
where $\Fc$ is the set of all possible density generators of elliptical distributions. 
Note that any $g$ in $\Fc$ has to satisfy~(\ref{cond_dens_gen}), but not~(\ref{cond_density_fg}).

\mds

In our case, we do not have access to a dataset following an elliptical distribution. Assume that we observe a dataset $\U_1, \dots, \U_n$ following a meta-elliptical copula $\Mc\Ec(\Sigmabf,
g)$.
If we knew the true generator $g$, we could compute the univariate quantile function $Q_g$, and, as a consequence, 
$\vec Q_g(\U_i) := \big( Q_g(U_{i,1}), \dots, Q_g(U_{i,d}) \big)^\top \sim \Ec_d(\0, \Sigmabf, g)$ for every $i\in \{1, \dots, n\}$.
Therefore, we could define an ``oracle estimator'' of $g$ by
\begin{align}
    \hat g^{oracle}
    := A \big( \vec Q_g(\U_1), \dots, \vec Q_g(\U_n) \big)
    \color{black}{=:A_n(g)}.
    \label{eq:oracle_estimator_g}
\end{align}
In practice, two issues arise that prevent us from using this oracle estimator.
First, we do not have access to the true distributions $F_1, \dots, F_d$, but only to empirical {\color{black}cdfs} $\hat F_1, \dots,
\hat F_d$. As usual, it is possible to replace the ``unobservable'' realizations $U_{i,k} = F_k(X_{i,k})$ by pseudo-observations $\hat U_{i,k} = \hat F_k(X_{i,k})$ in
Eq.~(\ref{eq:oracle_estimator_g}).
Second and more importantly, we need $Q_g$, i.e., $g$ itself, to compute the oracle estimator $\hat g^{oracle}$. This looks like an impossible task.

\mds

To solve the problem, we propose an iterative algorithm as follows. We fix a first estimate $\hat g^{(0)}$ of $g$, so that  we can compute
$\hat \Z^{(1)}_i = \vec Q_{\hat g^{(0)}} (\hat \U_i),$ for every $i\in \{1, \dots, n\}$.
From this first guess, we can compute an estimator
$\hat g^{(1)} := A \big( \hat \Z^{(1)}_{1}, \dots, \hat \Z^{(1)}_{n} \big)$.
Note that this estimator should be normalized in order to satisfy the necessary condition that is related to the identifiability of $g$. At this stage, we impose Condition~(\ref{cond_density_fg}), for a fixed constant $b$, and invoke Algorithm~\ref{algo:normalization_g}.
Iteratively, for any $N \in \Nb$, we define
$\hat \Z^{(N)}_i = \vec Q_{\hat g^{(N-1)}} (\hat \U_i)$
and $\hat g^{(N)} := A \big( \hat \Z^{(N)}_{1}, \dots, \hat \Z^{(N)}_{n} \big)$.
This procedure MECIP is detailed in Algorithm~\ref{algo:basic_iteration_estimation_elliptical} below.
See Fig.~\ref{fig:simplified_flowchart} too. 

\mds 

{\textcolor{black}{To summarize, for a fixed sample size $n$, the latter recursive algorithm is classical in the domain of fixed points analysis: we assume there exists a generator such that $g_n^*=A_n(g_n^*)$ and we approximate $g_n^*$ by a recursion $\hat g_n^{N+1}=A_n\big(\hat g_n^{N}\big)$. When $N\rightarrow \infty$, we hope that $\hat g_n^{N} \simeq g_n^*$. 
This is the case if $A_n$ is a contraction (Banach Contraction Principle), but the latter property is not guaranteed.
Moreover, with a large $n$, we hope that $g_n^*$ tends to the true underlying generator $g$, because of the consistency 
of the estimation procedure $A_n$. In other words, our algorithm is a mix between fixed point search procedures and classical nonparametric inference.
The theoretical study of its convergence properties appears as particularly complex, due to its multiple nonlinear stages.   
}}

\mds

Hereafter, $A$ will be chosen as Liebscher's estimation procedure~\cite{liebscher2005_semiparametric} to evaluate the generator of an elliptical distribution. 
It requires the introduction of 
the instrumental map $\psi_a$ defined by $\psi_a(x) := -a + (a^{d/2}+x^{d/2})^{2/d}$ for any $x\geq 0$ and some constant $a>0$. 
Since Liebscher's method is non-parametric, we need a usual univariate kernel $K$ (here, the Gaussian kernel) and a bandwidth $h=h_n$, $h_n\rightarrow 0$ when $n$ tends {\color{black}to infinity}. See the exact formula of Liebscher's estimator below, in Algorithm~\ref{algo:basic_iteration_estimation_elliptical}.

\mds

We will propose three ways of initializing the algorithm:
\begin{itemize}
\item[(i)] the ``Gaussian'' initialization, where $\hat g^{(0)} := e^{-x}$, suitably normalized (with the notations of Algorithm~\ref{algo:normalization_g}, $\Ic_1=\Gamma(d/2)$ and $\Ic_2=\Gamma(d/2-1/2)$); 
\item[(ii)] the ``identity'' initialization, in which $\hat g^{(0)} := A \big( (\hat U_{i,j})_{1 \leq i \leq n, \, 1 \leq j \leq d} \big)$ as if, in lines 8-9 of Algorithm~\ref{algo:basic_iteration_estimation_elliptical}, the quantile function $Q_{\hat g^{(N-1)}}$ were replaced by the identity map;
\item[(iii)] the ``$A\sim \Phi^{-1}$'' initialization, where $\hat g^{(0)} := A \Big( \big(\Phi^{-1}(\hat U_{i,j})\big)_{1 \leq i \leq n, \, 1 \leq j \leq d} \Big)$ and $\Phi$ denotes the cdf of a Gaussian $\Nc(0,1)$ distribution.
\end{itemize}
Solution (ii) may seem a brutal approximation and can be considered as an ``uninformative prior''. 
It can actually be well-suited, compared to (i), when the true generator is far from the Gaussian generator: see Fig.~\ref{fig:MSE_different_h_g}.
The main motivation of (iii) is to put $U$ whose support is $[0,1]$ back on the whole real line $\Rb$, through a usual numerical trick.

\mds

The implementation of the algorithm of MECIP is available in the \texttt{R} package \texttt{ElliptCopulas}~\cite{package_ElliptCopulas}.

\tikzstyle{block} = [rectangle, draw, align=center]
\tikzstyle{arr} = [-stealth]
\tikzstyle{arr_iter} = [-stealth, blue, thick]
\tikzstyle{arr_startEnd} = [-stealth, blue]

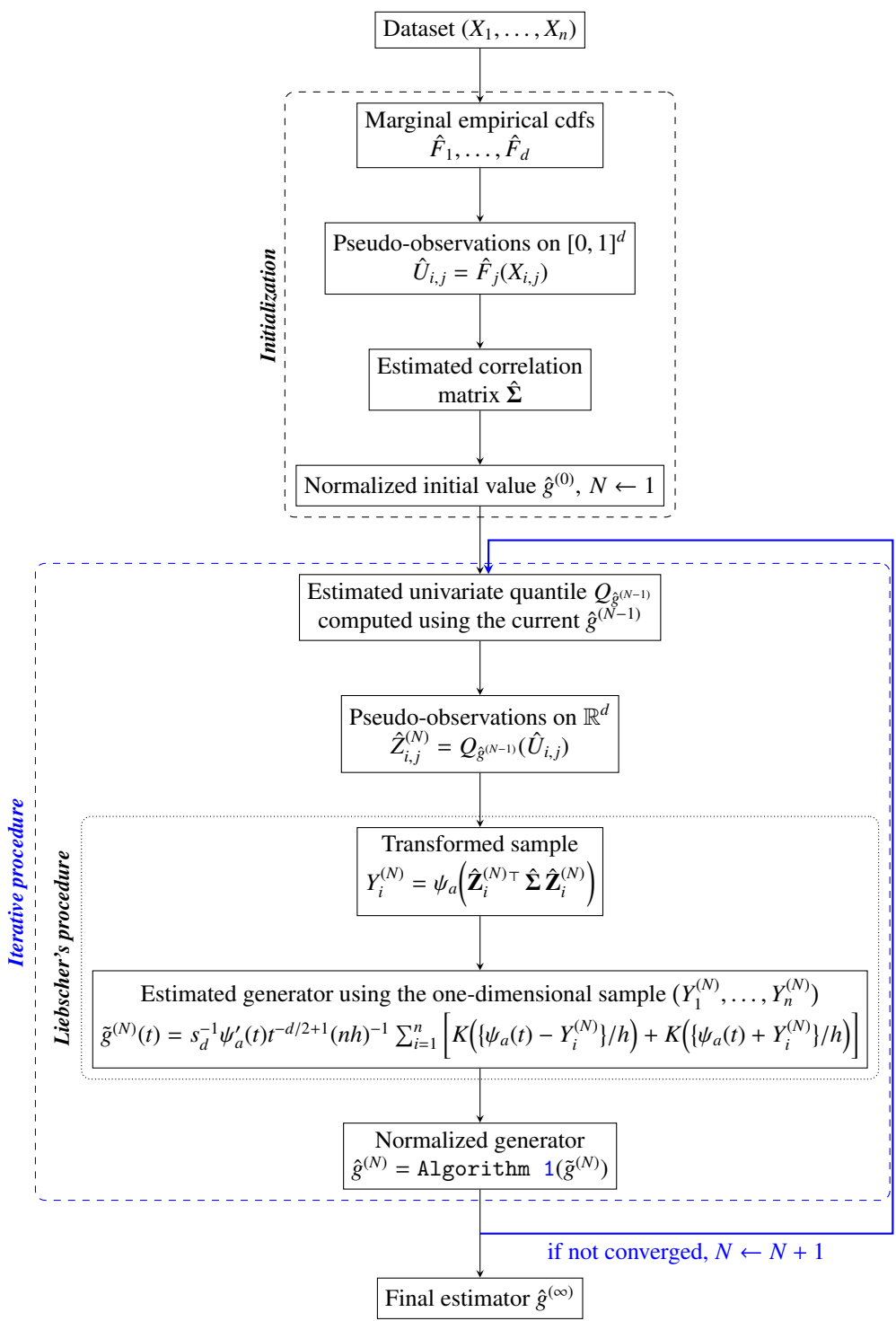
\begin{figure}[htb]
    \centering
    \begin{tikzpicture}[node distance = 0.8cm]
    
    \node[block] (dataset) {Dataset $(X_1, \dots, X_n)$};
    
    \node[block, below=of dataset.south] (empcdf) {Marginal empirical cdfs \\ $\hat F_1, \dots, \hat F_d$};
    
    \node[block, below=of empcdf.south] (pseudoobs) {Pseudo-observations on $[0,1]^d$ \\
    $\hat U_{i,j} = \hat F_j(X_{i,j})$};
    
    \node[block, below=of pseudoobs.south] (corrmat) {Estimated correlation \\ matrix $\hat \Sigmabf$};
    
    \node[block, below=of corrmat.south] (init_g) {Normalized initial value $\hat g^{(0)}, \, N \leftarrow 1$};
    
    \node[block, below=of init_g.south, yshift=-0.2cm] (univQuantiles) {Estimated univariate quantile $Q_{\hat g^{(N-1)}}$ \\
    computed using the current $\hat g^{(N-1)}$};
    
    \node[block, below=of univQuantiles.south] (pseudoobsRd) {Pseudo-observations on $\Rb^d$ \\ $\hat Z_{i,j}^{(N)} = Q_{\hat g^{(N-1)}}(\hat U_{i,j})$};
    
    \node[block, below=of pseudoobsRd.south] (transSample) {Transformed sample \\
    $Y_i^{(N)} = \psi_a \bigg(\hat \Z_{i}^{(N)}{}^\top \,\hat \Sigmabf  \, \hat \Z_{i}^{(N)} \bigg)$
    };
    
    \node[block, below=of transSample.south] (esti_g) {Estimated generator using the one-dimensional sample $\big(Y_1^{(N)}, \dots, Y_n^{(N)}\big)$ \\
    $\tilde g^{(N)}(t) = s_d^{-1} \psi_a'(t) t^{-d/2+1} (n h)^{-1} \sum_{i=1}^n \bigg[
    K \Big( \big\{\psi_a(t) - Y_i^{(N)} \big\}/h\Big)
    + K \Big( \big\{\psi_a(t) + Y_i^{(N)} \big\}/h\Big)
    \bigg]$
    };
    
    \node[block, below=of esti_g.south] (normalized_g) {Normalized generator \\
    $\hat g^{(N)} = \texttt{Algorithm \ref{algo:normalization_g}} (\tilde g^{(N)})$
    };
    
    \node[block, below=of normalized_g.south, yshift=-0.5cm] (final_g) {Final estimator $\hat g^{(\infty)}$};
    
    \path[draw, arr] (dataset) -- (empcdf);
    \path[draw, arr] (empcdf) -- (pseudoobs);
    \path[draw, arr] (pseudoobs) -- (corrmat);
    \path[draw, arr] (corrmat) -- (init_g);
    
    \path[draw, arr] (init_g) -- (univQuantiles)
    node [pos=0.5] (emptyFinal_4) {};
    \path[draw, arr] (univQuantiles) -- (pseudoobsRd);
    \path[draw, arr] (pseudoobsRd) -- (transSample);
    \path[draw, arr] (transSample) -- (esti_g);
    \path[draw, arr] (esti_g) -- (normalized_g);
    
    \path[draw, arr] (normalized_g) -- (final_g) 
    node [pos=0.5] (emptyFinal) {};
    
    \node[right= of emptyFinal, xshift = 5cm] (emptyFinal_2) {};
    \node[right= of emptyFinal_4, xshift = 5cm] (emptyFinal_3) {};
    \path[draw, arr, blue, thick] (emptyFinal.center) -- (emptyFinal_2.center)
    node [pos=0.5, below] {if not converged, $N \leftarrow N+1$}
    --  (emptyFinal_3.center) -- (emptyFinal_4.east) -- (univQuantiles.76);
    
    \node (initbox) [draw, rounded corners, dashed, inner sep=1ex, 
        fit=(empcdf) (pseudoobs) (corrmat) (init_g),
        label={[node font=\small\itshape\bfseries,rotate=90,anchor=south]left: Initialization}
        ]
        {};
                
    \node (lieb) [draw, rounded corners, densely dotted, black, inner sep=1ex, 
        fit=(transSample) (esti_g),
        label={[node font=\small\itshape\bfseries,rotate=90,anchor=south, name=labelLieb]left: Liebscher's procedure}
        ]
        {};
                
    \node (iteration) [draw, blue, rounded corners, dashed, inner sep=1ex, 
        fit=(lieb) (normalized_g) (univQuantiles) (labelLieb),
        label={[node font=\small\itshape\bfseries,rotate=90,anchor=south,color=blue]left: Iterative procedure}
        ]
        {};
                
    \end{tikzpicture}
    
    \caption{\textnormal{Simplified flowchart of the iterative estimation procedure MECIP.}}
    \label{fig:simplified_flowchart}
\end{figure}

\begin{algorithm}[htbp]
\label{algo:basic_iteration_estimation_elliptical}
\SetAlgoLined
    \vspace{0.1cm}
    \KwIn{A dataset $(\X_1, \dots ,\X_n)$, $n > 0$.}
    \KwIn{An estimation method $A$ for elliptical distribution density generators.}
    \For{$j \leftarrow 1$ \KwTo $d$} {
        Compute the empirical cdf $\hat F_j$ and the pseudo-observations $\hat U_{i,j} := \hat F_j(X_{i,j})$ for all $i\in \{1, \dots, n\}$\;
    }
    Compute an estimator of the correlation matrix $\hat \Sigmabf$ of the elliptical copula using Kendall's taus \;
    Initialize $N:=1$. Initialize $g$ to a value $\hat g^{(0)}$ \;
    Normalize $\hat g^{(0)} := \texttt{Algorithm \ref{algo:normalization_g}} (\tilde g^{(0)})$ \;
    \Repeat(){convergence of $\hat g$}{
        Compute the univariate quantile function
        $Q_{\hat g^{(N-1)}}$ associated with the elliptical distribution~$\Ec_d(\0, \I_d, \hat g^{(N-1)})$;
        
        For every $i\in \{1, \dots, n\}$ and $j\in \{1, \dots, d\}$, compute $\hat Z_{i,j}^{(N)}
        := Q_{\hat g^{(N-1)}} (\hat U_{i,j})$\;
        Update $\hat g^{(N)}
        := A(\hat \Z_1^{(N)}, \dots, \hat \Z_n^{(N)} \, ; \, \hat \Sigmabf)$, corresponding to the following step when $A$ is Liebscher's procedure with a given $a>0$\;
         \Begin{
             \tcc{Liebscher's procedure}
             For $i\in \{1, \dots, n\}$, let
             $Y_i = \psi_a \Big(\hat \Z_{i}^{(N)} {}^\top
             \hat \Sigmabf^{-1} \hat \Z_{i}^{(N)} \Big)$ \;
             Let $\tilde g^{(N)}(t) = s_d^{-1} \psi_a'(t) t^{-d/2+1} (n h)^{-1} \sum_{i=1}^n \bigg[
            K \Big( \big\{\psi_a(t) - Y_i\big\}/h\Big)
             + K \Big( \big\{\psi_a(t) + Y_i\big\}/h\Big)
             \bigg]$ \;
 }
        Normalize $\hat g^{(N)} := \texttt{Algorithm \ref{algo:normalization_g}} (\tilde g^{(N)})$ \;
        Update $N = N+1$ \;
    }

    \KwOut{A normalized estimator $\hat g^{(\infty)}$ of $g$.}
\caption{Basic iteration-based procedure MECIP for the estimation of the elliptical copula density generator~$g$ with a given method $A$}
\end{algorithm}

{\color{black}

\subsection{Adjustments of the iterative algorithm MECIP in the presence of missing values}

Whenever a dataset contains missing values, the previous numerical procedure can be adapted to estimate the correlation matrix $\Sigmabf$ and the generator of the underlying elliptical copula. We consider the simplest case of missing at random observations. Note that many other missing patterns may exist, but a complete treatment of these cases is left for future work.

\mds

When missing values arise, the previous Algorithm~\ref{algo:basic_iteration_estimation_elliptical} will be adjusted as follows:
\begin{enumerate}
    \item Each empirical cdf $\hat F_j$ is estimated using all non-missing observations for the $j$-th variable.
    
    \item Pseudo-observations $\hat U_{i,j} := \hat F_j(X_{i,j})$ are defined as ``\NA'' whenever $X_{i,j}=\NA$ (i.e., is missing).
    
    \item Kendall's taus are estimated using pairwise complete observations. In other words, for two variables $1 \leq j_1 \neq j_2 \leq d$, the Kendall's tau between $X_{j_1}$ and $X_{j_2}$ is estimated using the set of observations $\big\{i\in\{1, \dots, n\}: X_{i,j_1} \neq \NA \text{ and } X_{i,j_2} \neq \NA \big\}$.
    
    \mds
    
    If the correlation matrix is not positive semi-definite, it is projected on the (convex) set of positive semi-definite matrices using the \texttt{R} function \texttt{nearPD}, which implements the method proposed in \cite{higham2002nearest}. See alternative methods in~\cite{higham2016restoring, qi2006quadratically}, among others. 
    
    \item The pseudo-observations
    $\hat Z_{i,j}^{(N)} := Q_{\hat g^{(N-1)}} (\hat U_{i,j})$ at iteration $N$ are defined to be $\NA$ whenever $\hat U_{i,j}=\NA$.
    
    \item At each step of the main loop, for every $i\in \{1,\dots,n\}$ such that some $\hat Z_{i,j}^{(N)}$ is \NA, we complete the vector $\hat \Z_i^{(N)}$ in the following way:
    let $miss(i)=\big\{j \in \{1, \dots,d\}: \hat Z_{i,j}^{(N)} = \NA\big\}$ be the set of the missing components for the $i$-th observation.
    The non-missing part of $\hat \Z_i^{(N)}$, denoted $\hat \Z_{i, -miss(i)}^{(N)}$, is left unchanged.
    If we knew the true generator $g$ and the correlation matrix $\Sigmabf$, we would use the non-missing part $\hat \Z_{i, -miss(i)}^{(N)}$ of the random vector $\hat \Z_{i}^{(N)}$ to complete the other entries by using their conditional law. 
    Indeed, if a vector $\Z$ follows an elliptical distribution $\Ec_d(\mubf,\Sigmabf,g)$, then, for any subset $I \subset \{1, \dots, d\}$, the law of $\Z_I$ given $(\Z_{-I} = \z_{-I})$ is still elliptical $\Ec_{|I|}(\mubf_{\z_{-I}},\Sigmabf_{\z_{-I}},g_{\z_{-I}})$.
    Some explicit expressions for these three conditional parameters are given in \cite[Corollary 5]{cambanis1981_theory}.
    Therefore, an ``oracle'' way of generating $\hat \Z_{i, miss(i)}^{(N)}$ is to draw
    \begin{align*}
        \hat \Z_{i, miss(i)}^{(N)}
        \sim \Ec_{|miss(i)|} \big(
        \mubf_{\hat \Z_{i, -miss(i)}^{(N)}},
        \Sigmabf_{\hat \Z_{i, -miss(i)}^{(N)}},
        g_{\hat \Z_{i, -miss(i)}^{(N)}} \big),
    \end{align*}
    neglecting the fact that $\hat \Z_i^{(N)}$ follows only approximately an elliptical distribution (unless $\hat g^{(N)} = g$, which is rather unlikely).
    However, $\Sigmabf$ and $g$ are unknown; then, we propose to replace $\Sigmabf$ by its empirical counterpart and $g$ by its most recent estimate $\tilde  g^{(N)}$. Finally, we obtain the updated ``feasible'' generating formula
    \begin{align}
        \hat \Z_{i, miss(i)}^{(N)}
        \sim \Ec_{|miss(i)|} \big(
        \hat \mubf_{\hat \Z_{i, -miss(i)}^{(N)}},
        \hat \Sigmabf_{\hat \Z_{i, -miss(i)}^{(N)}},
        \hat g_{\hat \Z_{i, -miss(i)}^{(N)}}^{(N)} \big),
        \label{eq:completion_Z_missing}
    \end{align}
    for some approximate conditional mean $\hat \mubf_{\hat \Z_{i, -miss(i)}^{(N)}}$ based on $\hat \Sigmabf$, some approximate conditional correlation matrix $\hat \Sigmabf_{\hat \Z_{i, -miss(i)}^{(N)}}$ based on $\hat \Sigmabf$ and some approximate conditional generator based on $\tilde g^{(N)}$ and on $\hat \Sigmabf$.
\end{enumerate}
}

\begin{algorithm}[htbp]
\label{algo:basic_iteration_estimation_elliptical_missing values}
\SetAlgoLined
    \vspace{0.1cm}
    \KwIn{A dataset $(\X_1, \dots ,\X_n)$, $n > 0$.}
    \KwIn{An estimation method $A$ for elliptical distribution density generators.}
    \For{$j \leftarrow 1$ \KwTo $d$} {
        Compute the empirical cdf $\hat F_j$ {\color{black}using available data for the $j$-th variable} and the pseudo-observations $\hat U_{i,j} := \hat F_j(X_{i,j})$ for all $i\in \{1, \dots, n\}$\;
    }
    Compute an estimator of the correlation matrix $\hat \Sigmabf$ of the elliptical copula using Kendall's taus {\color{black}estimated on pairwise complete observations} \;
    Initialize $N:=1$. Initialize $g$ to a value $\tilde g^{(0)}$ \;
    Normalize $\hat g^{(0)} := \texttt{Algorithm \ref{algo:normalization_g}} (\tilde g^{(0)})$ \;
    \Repeat(){convergence of $\hat g$}{
        Compute the univariate quantile function
        $Q_{\hat g^{(N-1)}}$ associated with the elliptical distribution~$\Ec_d(\0, \I_d, \hat g^{(N-1)})$;
        
        For every $i\in \{1, \dots, n\}$ and $j\in \{1, \dots, d\}$, compute $\hat Z_{i,j}^{(N)}
        := Q_{\hat g^{(N-1)}} (\hat U_{i,j})$ \;
        {\color{black}\ForIf{$i\in \{1, \dots, n\}$:}
        {$\Z_i$ contains missing values at entries $miss(i)$,} 
        {Simulate the missing part $\Z_{i,miss(i)}^{(N)}$ of $\Z_i^{(N)}$ using Eq.~\eqref{eq:completion_Z_missing}
        }
        }
        Update $\hat g^{(N)}
        := A(\hat \Z_1^{(N)}, \dots, \hat \Z_n^{(N)} \, ; \, \hat \Sigmabf)$, corresponding to the following step when $A$ is Liebscher's procedure with a given $a>0$ \;
         \Begin{
             \tcc{Liebscher's procedure}
             For $i\in \{1, \dots, n\}$, let
             $Y_i = \psi_a \Big(\hat \Z_{i}^{(N)} {}^\top
             \hat \Sigmabf^{-1} \hat \Z_{i}^{(N)} \Big)$ \;
             Let $\tilde g^{(N)}(t) = s_d^{-1} \psi_a'(t) t^{-d/2+1} (n h)^{-1} \sum_{i=1}^n \bigg[ 
            K \Big( \big\{\psi_a(t) - Y_i\big\}/h\Big)
             + K \Big( \big\{\psi_a(t) + Y_i\big\}/h\Big)
             \bigg]$ \;
 }
        Normalize $\hat g^{(N)} := \texttt{Algorithm \ref{algo:normalization_g}} (\tilde g^{(N)})$ \;
        Update $N = N+1$ \;
    }

    \KwOut{A normalized estimator $\hat g^{(\infty)}$ of $g$.}
\caption{Improved version of the iteration-based procedure MECIP for the estimation of the elliptical copula density generator~$g$ with a given method $A$}
\end{algorithm}








\FloatBarrier

\section{Numerical results for the iteration-based method MECIP}
\label{numerical_results}

\subsection{Simulation study in dimension 2}

For this simulation study, fix the dimension $d=2$ and the sample size $n=1000$. 
The values of the estimated generators are calculated on a grid of the interval $[0,10]$ with the step size $0.005$.
The correlation matrix is chosen as $\Sigmabf = \begin{pmatrix} 1 & 0.2 \\ 0.2 & 1\end{pmatrix}$. We use uniform marginal distributions, but still estimate them nonparametrically using the empirical distribution function as if they were unknown.

\mds

We consider the normalised versions of six possible generators:
$g(x) = 1/(1+x^2)$, $g(x) = e^{-x}$, $g(x) = e^{-x} + \text{bump}(x)$, 
$g(x) = e^{-x} + e^{-x/3} \cos^2(x)$, $g(x) = x/(1+x^3)$, and $g(x) = x^2 e^{-x^2}$,
where $\text{bump}(x) = \1\{x \in [1, 1+\pi]\} (x-1) (1+\pi-x) \sin(x-1)$ is a smooth function supported on $[1, 1+\pi]$.
The estimated generators obtained with the iterative method (Algorithm~\ref{algo:basic_iteration_estimation_elliptical} using Liebscher's procedure) are plotted on Fig.~\ref{fig:plots_iterations}. In general and after less than $N=10$ iterations, our estimated generators yield convenient approximations of the true underlying generators, even if $g(0)=0$ (a case that was excluded by Condition~\ref{reg_generator}). Nonetheless, when $g$ is highly non monotonic, as 
for ``double-bump'' generators, the iterative algorithm is less performing and larger sample sizes are required. 
{\color{black}Note that, when $d=2$, the parameter $a$ has no influence on the result, since for any $x>0$, $\psi_a(x) := -a + (a^{d/2}+x^{d/2})^{2/d} = x$.}

\begin{figure}
    \includegraphics[width = 18cm]{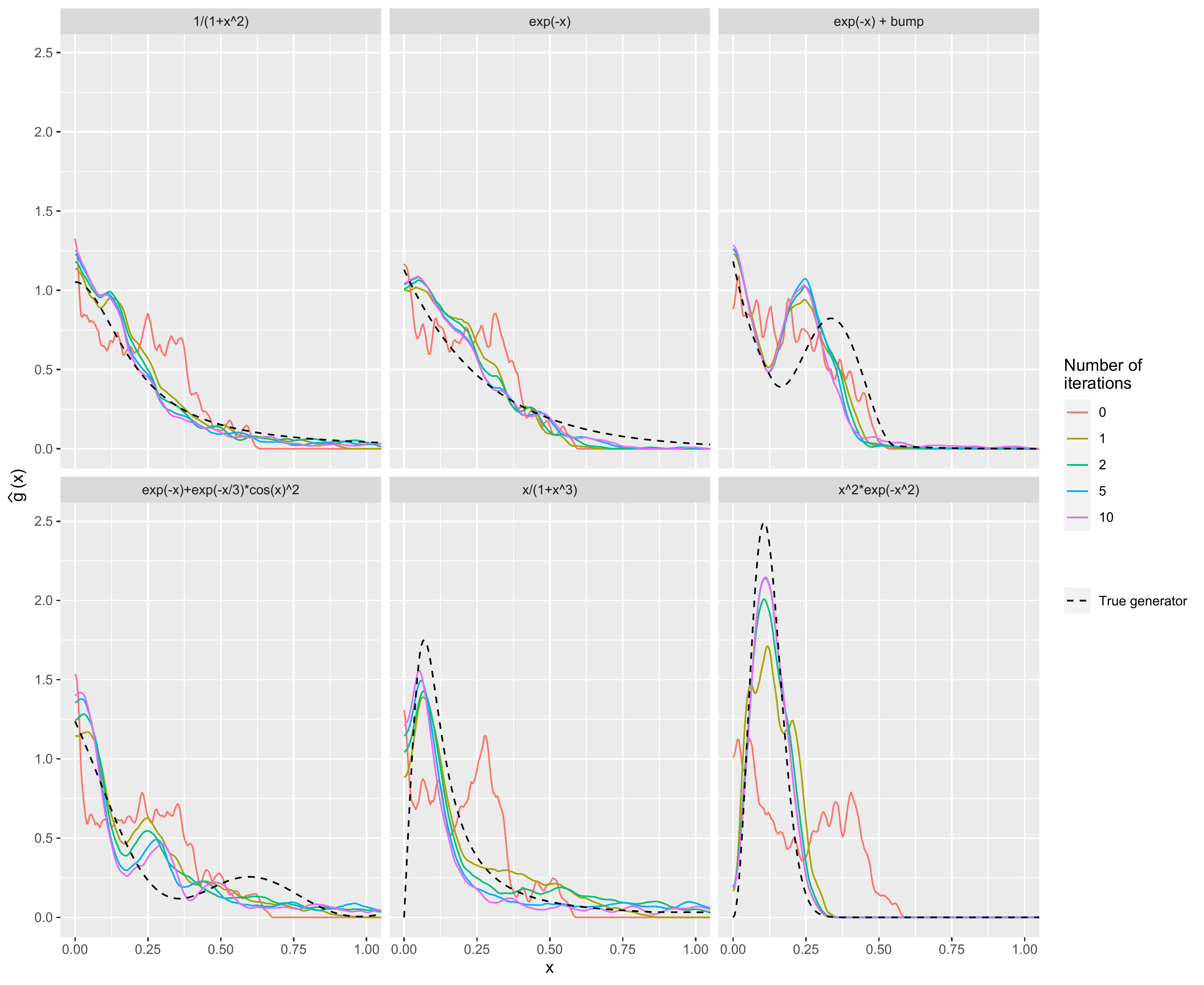}
    \caption{\textnormal{Estimated generators by the iterative method MECIP, with $n=1000$, $d=2$, $\rho=0.2$, $a=1$, $h=0.05$ and starting point = ``identity''.}}
    \label{fig:plots_iterations}
\end{figure}

\mds

In Fig.~\ref{fig:MSE_different_h_g}, the mean integrated squared error $\text{MISE}(\hat g_h) := E[||\hat g_h - g||_2^2]$ of our iterative estimator is plotted as a function of the bandwidth $h$, for different true generators $g$ and initialization strategies.
These MISE are computed using $N$ equal to ten and $100$ replications of each experiment. As expected, a clear-cut optimal bandwidth can be empirically identified for most generators. 

\mds

Restricting ourselves to the case of the Gaussian generator, we then study the joint influence of the sample size $n$ and of the bandwidth $h$ on the MISE of our estimators, using the ``identity'' initialization: see 
Fig.~\ref{fig:MSE_different_h_n_V3}. 
We find the same behaviors as for usual kernel-based estimators: empirically, the optimal bandwidths are "closely" linear functions of $\ln(n)$.
{\color{black}The computation time for the three initialization methods are compared in Fig.~\ref{fig:CompTime_startPoint}. They are mostly similar.
The ``Gaussian'' initialization is the fastest method as it is not data-dependent.}

\begin{figure}
    \includegraphics[width = 18cm]{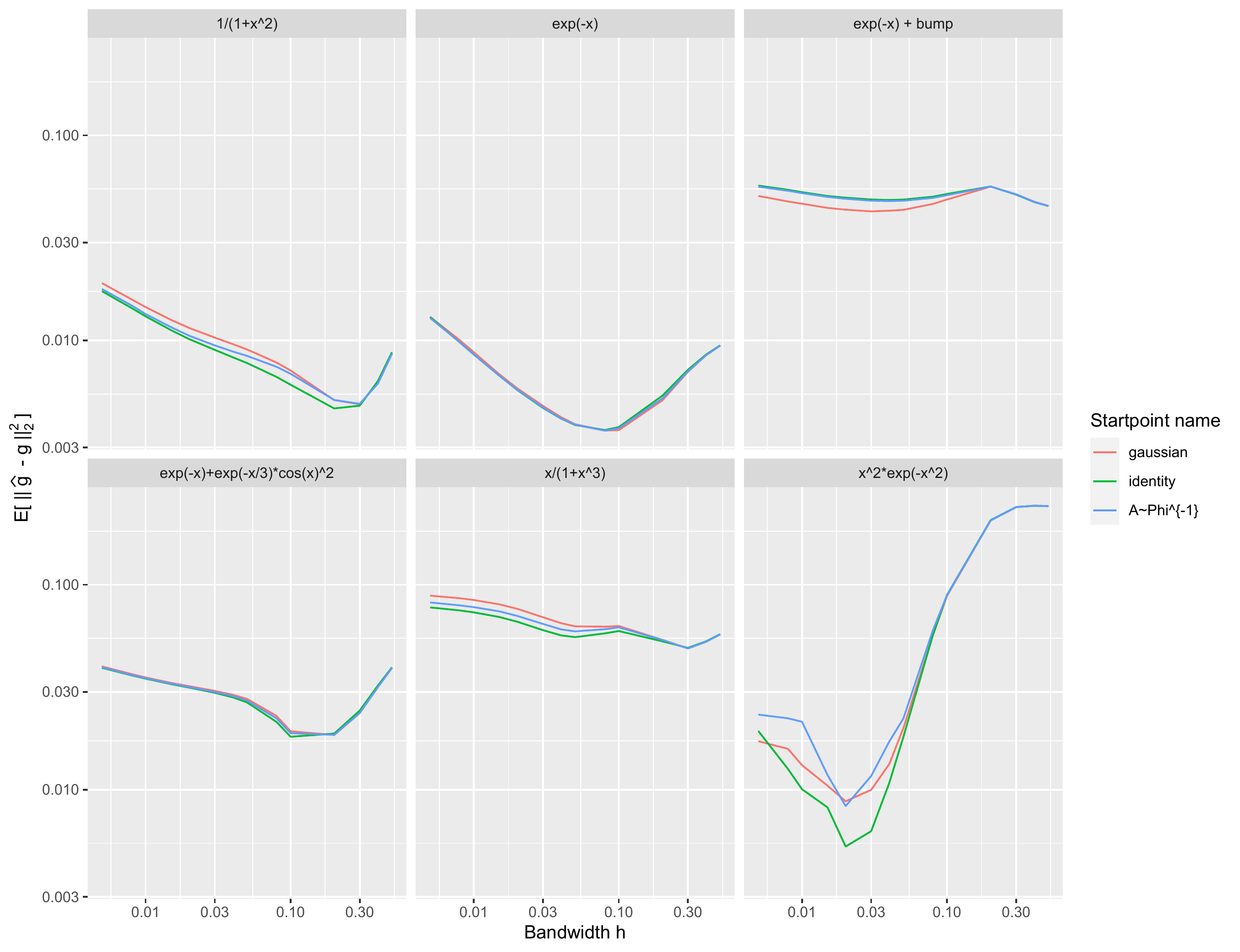}
    \caption{\textnormal{MISE of the estimates given by Algorithm~\ref{algo:basic_iteration_estimation_elliptical} for $n = 1000$, different choices of the generator $g$, the bandwidth $h$ and the initialization method.}}
    \label{fig:MSE_different_h_g}
\end{figure}

\begin{figure}
    \includegraphics[width = \textwidth]{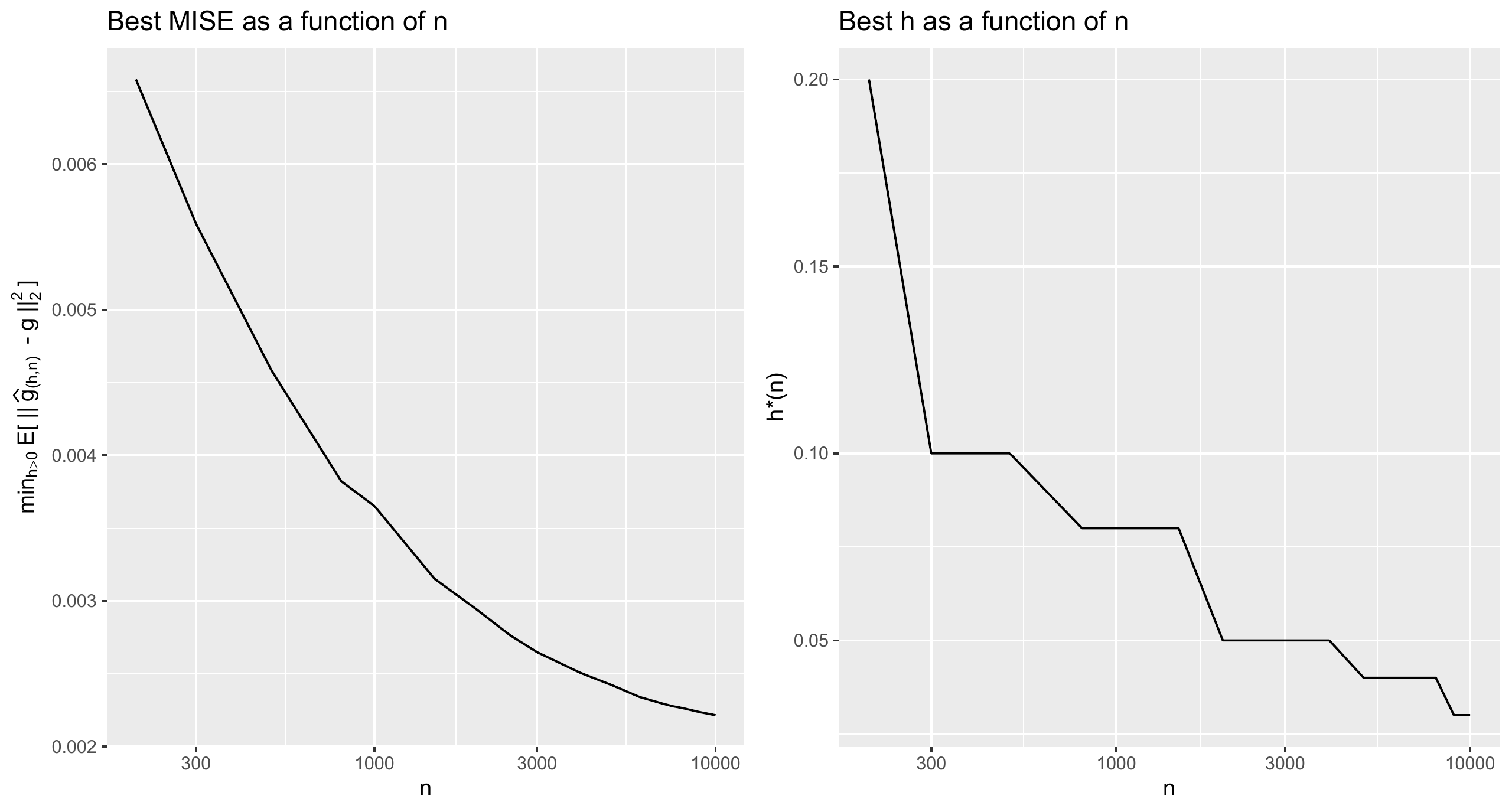}
   \caption{\textnormal{MISE of the estimate given by Algorithm~\ref{algo:basic_iteration_estimation_elliptical} as a function of $n$, for the best bandwidth $h=h^*(n)$, and $h^*(n)$ as a function of $n$. Both plots are in semi-log scale. The true generator is the Gaussian $g(x) = e^{-x}$ with the initialization ``identity''.}}
    \label{fig:MSE_different_h_n_V3}
\end{figure}

\subsection{Simulation study for higher dimensions}

Here, we consider the same sample size $n=1000$ as before, but the dimension $d$ varies between $3$ and $11$.
The correlation matrix is chosen as $\Sigma_{i,j} = 0.2$ when $i \neq j$. For the Gaussian generator, we then study the performance of our algorithm as a function of the tuning parameters $a$ and $h$. 
The results are displayed in Fig.~\ref{fig:MSE_a_h_d}.
We observe that the MSE increases with the dimension, even for the best choices of the tuning parameters $a$ and $h$.
When $d\geq 5$, avoid choosing $a$ less than one, while the influence of the bandwidth $h$ seems to be less crucial.

\mds

Computation times increase only slowly with the dimension $d$: see Fig.~\ref{fig:CompTime_dim}. This is because the generator is a 
univariate function regardless the dimension of the random vector. Therefore, most steps in our algorithms are invariant with respect to the dimension $d$, except the transformation of the sample $Y_i := \psi_a \big(\hat \Z_{i}^\top \,\hat \Sigmabf \, \hat \Z_{i} \big)$. The latter step costs at most $(4d^2+2d+3)n$ elementary operations, a reasonable amount when $d$ is moderate.

\mds 

\begin{rem}
\textnormal{
Actually, it is possible to bypass the problem of high dimensions $d$. 
Indeed, any subvector of an elliptical distribution is itself elliptically distributed (see Eq.~\eqref{eq:gen_subvector} and the related discussion).
As a consequence, if the copula of a random vector $\Y$ is elliptical with generator $g_d$, then the copula of a subvector $\Y_{(m)}$ of $m$ components of $\Y$ is still an elliptical copula whose generator $g_m$ is given by Eq.~\eqref{eq:gen_subvector}.
Therefore, it is possible to estimate the generator of an elliptical copula by using only a sample of $m$-dimensional subvectors. 
By a numerical inversion of Eq.~\eqref{eq:gen_subvector}, one would get a generator that corresponds to the copula of the whole vector $\Y$. 
}
\end{rem}

\begin{figure}
    \includegraphics[width = \textwidth]{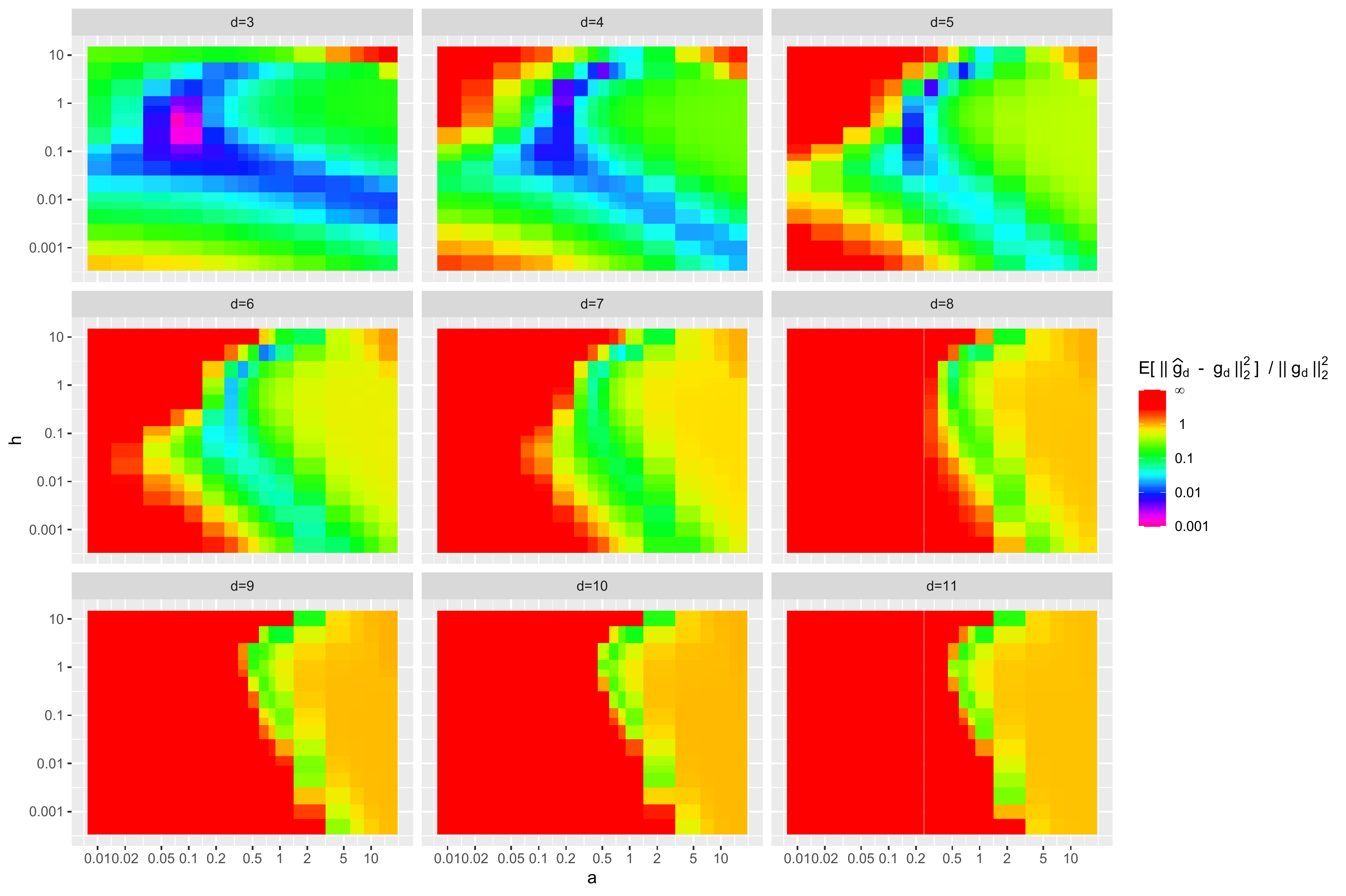}
   \caption{\textnormal{MISE of the estimate given by Algorithm~\ref{algo:basic_iteration_estimation_elliptical} as a function of $a, h$, for $n=1000$, and different choice of the dimension $d$. The true generator is the Gaussian $g(x) = e^{-x}$ with the initialization ``identity''.}}
    \label{fig:MSE_a_h_d}
\end{figure}

\begin{figure}[htb]
    \begin{minipage}[t]{0.442\textwidth}
        \begin{center}
 \includegraphics[width=\textwidth]{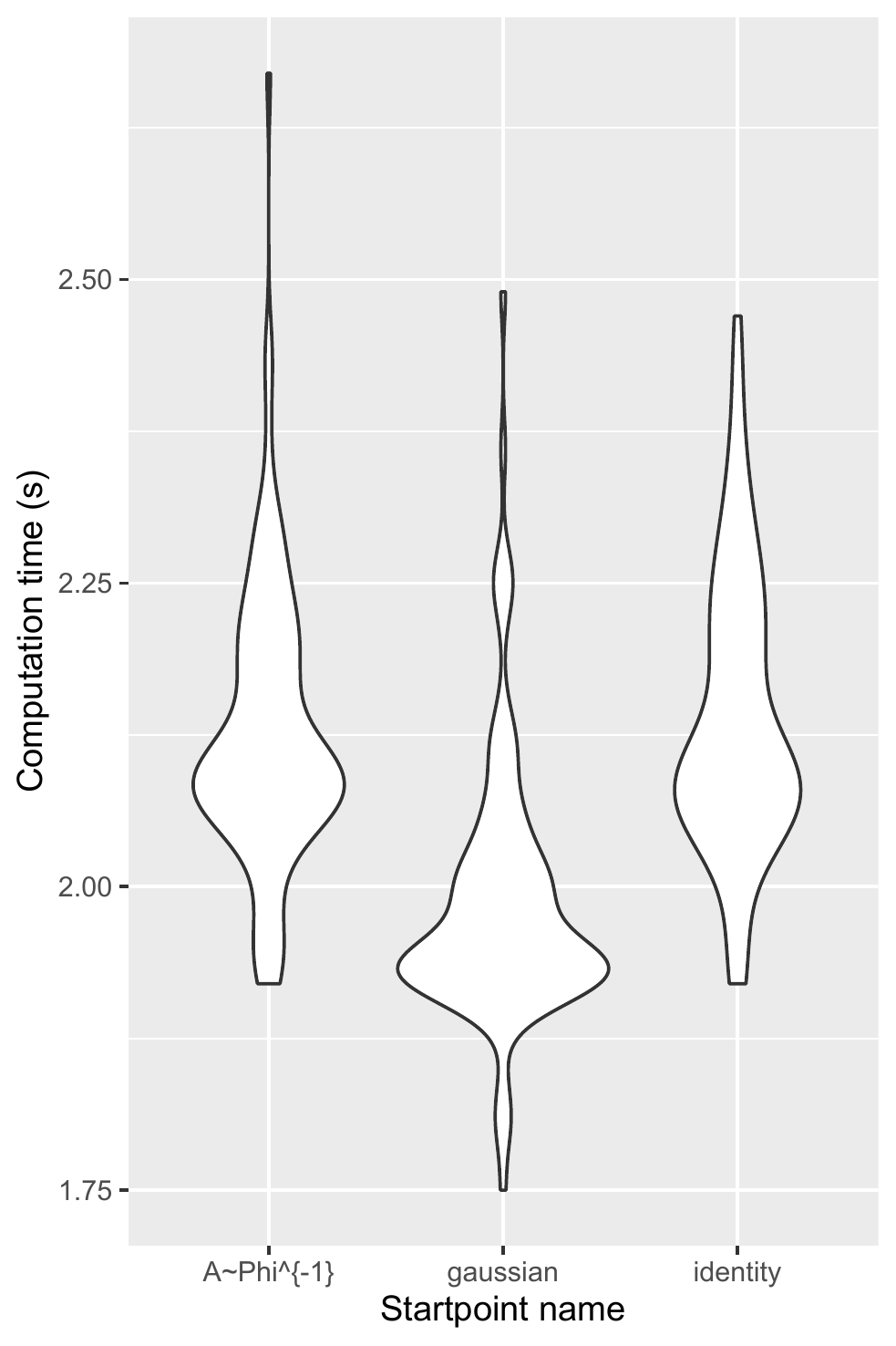}
 \caption{\textnormal{Density plot of the computation time as a function of the starting point, for $d=2$ and $n=1000$.}}
 \label{fig:CompTime_startPoint}
 \end{center}
    \end{minipage}
    \hfill
    \begin{minipage}[t]{0.53\textwidth}
       \begin{center}
 \includegraphics[width=\textwidth]{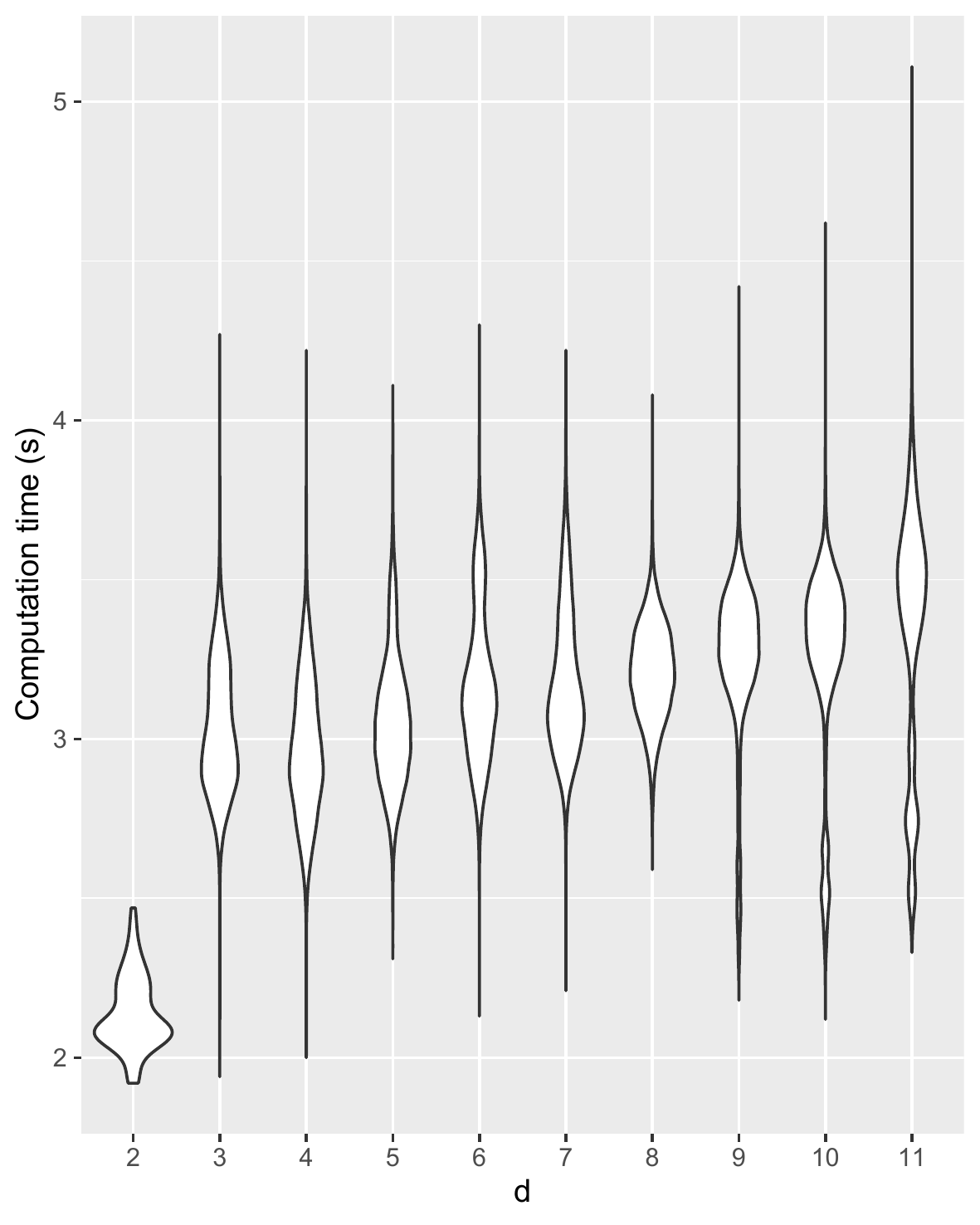}
 \caption{\textnormal{Density plot of the computation time as a function of the dimension, for the initialization method ``identity'' and $n=1000$.}}
 \label{fig:CompTime_dim}
 \end{center}
    \end{minipage}
\end{figure}

\subsection{When \texorpdfstring{$\Sigmabf$}{Sigma} is non exchangeable and almost non invertible}

We consider the same setting as in the previous section, and we aim for measuring the effect of the lack of \mbox{exchangeability} in the matrix $\Sigmabf$ on the estimation of $g$.
For this, two different frameworks are considered.
\begin{enumerate}
    \item[(a)] First framework: the dimension is $d=3$, $h=a=0.2$ and the correlation matrix is
    $\Sigmabf_{(3)}(\rho_{12}) = \begin{pmatrix}
    1 & \rho_{12} & 0.2 \\
    \rho_{12} & 1 & 0.2 \\
    0.2 & 0.2 & 1
    \end{pmatrix}$.
    
    \item[(b)] Second framework: the dimension is $d=10$, $a=1$, $h=0.1$ and the correlation matrix is
    \begin{align*}
        \Sigmabf_{(10)}(\rho_{12}) := \begin{pmatrix}
        1         & \rho_{12} & \rho_{12} & 0.2    & \cdots & 0.2    \\
        \rho_{12} & 1         & \rho_{12} & 0.2    & \cdots & 0.2    \\
        \rho_{12} & \rho_{12} & 1         & 0.2    & \cdots & 0.2    \\
        0.2       &  0.2      & 0.2       & 1      & \ddots & \vdots \\
        \vdots    &  \vdots   & \vdots    & \ddots & \ddots &  0.2   \\
        0.2       &  0.2      & 0.2       & \cdots & 0.2    &  1
        \end{pmatrix}
    \end{align*}
\end{enumerate}

Note that there exists a value $\underline{\rho_{12}}$ for which
$\Sigmabf_{(3)}(\rho_{12})$ and $\Sigmabf_{(10)}(\rho_{12})$ are positive semi-definite if and only if $\rho_{12}> \underline{\rho_{12}}$.
For the first and the second frameworks, $\underline{\rho_{12}} \approx -0.919$ and $\underline{\rho_{12}} \approx -0.432$ respectively.

\medskip

The MISE is computed for both frameworks as a function of $\rho_{12}$ and is displayed respectively in Fig.~\ref{fig:MSE_dim3_rho12} and~\ref{fig:MSE_dim10_rho12}. On these figures, the MISE stays stable whenever $\rho_{12}$ is not too close to the boundary value $\underline{\rho_{12}}$. When $\rho_{12}$ is close to $\underline{\rho_{12}}$, the estimator of the correlation matrix becomes unreliable, degrading the performance of the estimator $\hat g$. This deterioration is stronger in the second framework where $d=10$. Note that, as in the previous section, the performance for $d=10$ is always worse than for $d=3$ even far away from the boundary.

\begin{figure}[htb]
    \begin{minipage}[t]{0.48\textwidth}
        \begin{center}
 \includegraphics[width=\textwidth]{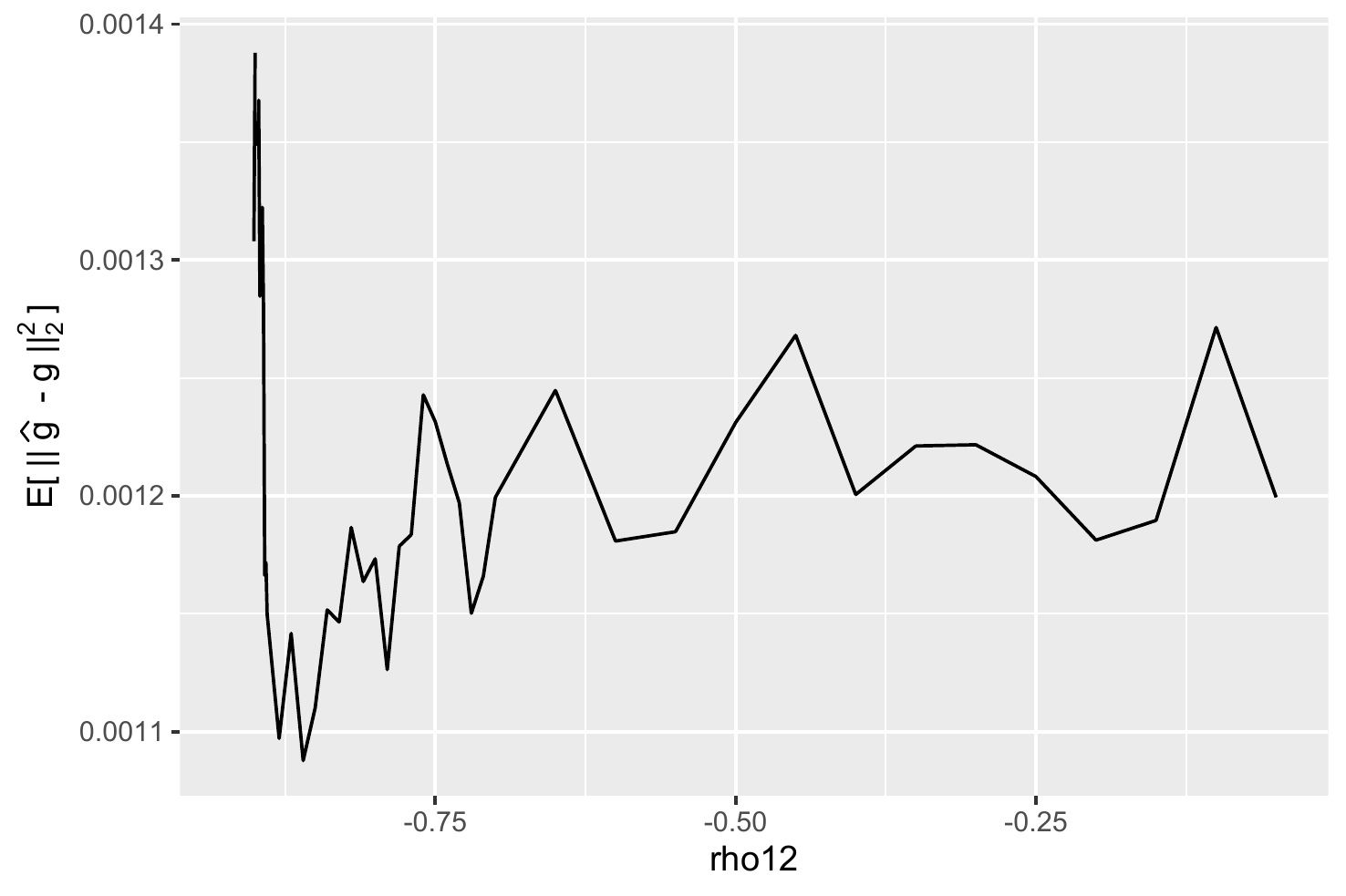}
 \caption{\textnormal{MISE of the estimate given by Algorithm~\ref{algo:basic_iteration_estimation_elliptical} as a function of $\rho_{12}$ for $d=3$ and $n=1000$, where the correlation matrix is $\Sigmabf_{(3)}(\rho_{12})$.}}
 \label{fig:MSE_dim3_rho12}
 \end{center}
    \end{minipage}
    \hfill
    \begin{minipage}[t]{0.48\textwidth}
       \begin{center}
 \includegraphics[width=\textwidth]{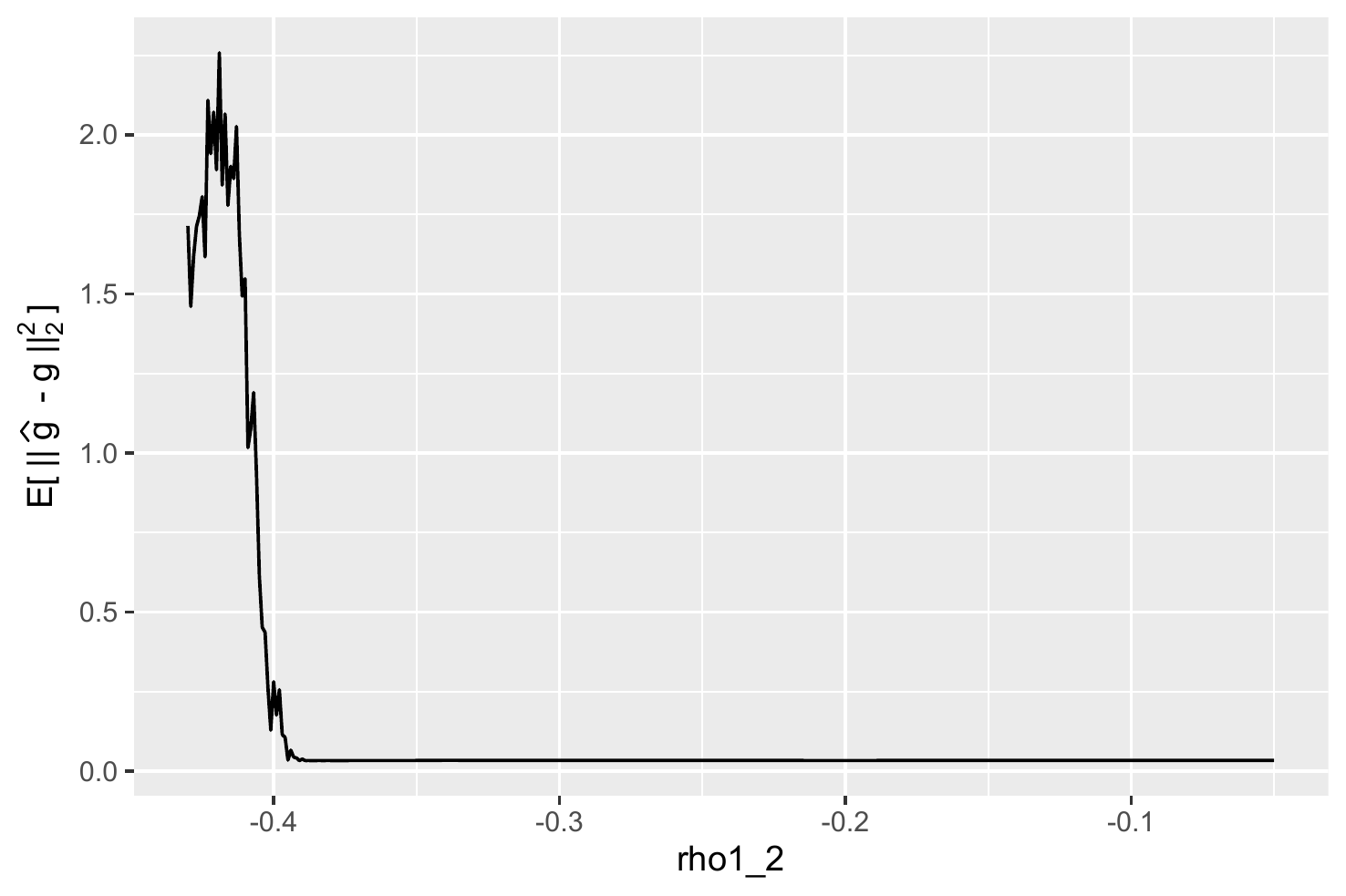}
 \caption{\textnormal{MISE of the estimate given by Algorithm~\ref{algo:basic_iteration_estimation_elliptical} as a function of $\rho_{12}$ for $d=10$ and $n=1000$, where the correlation matrix is $\Sigmabf_{(10)}(\rho_{12})$.}}
 \label{fig:MSE_dim10_rho12}
 \end{center}
    \end{minipage}
\end{figure}

\subsection{Simulation study with missing values}

In this case, we choose $n=1000$, $d=3$, and the same correlation matrix as before.
Contrary to the previous simulation experiments, we introduce some missing values in the dataset (represented by $\NA$ in the~\texttt{R} environment). 
The missing values are generated as follows:
we fix a number $N_{missing}$ of observations that are affected by missing values; we randomly draw a number $N_{missing,1}$ of observations (uniformly between $0$ and $N_{missing}$) for which a single component is missing; let $N_{missing,2} := N_{missing} - N_{missing,1}$ be the number of observations for which two components are missing. Because $d=3$, considering the case of three missing components does not make sense because this would induce an empty vector (which should rather induce a smaller sample size rather than a missing value issue if it were the case). We select $N_{missing,1}$ (respectively $N_{missing,2}$) observations at random, and replace 
the missing values by $\NA$.

\mds

For the estimation procedure, we choose $a = 0.08$ and $h=0.2$, the optimal choices according to Fig.~\ref{fig:MSE_a_h_d}.
The corresponding MISE are displayed in Fig.~\ref{fig:MSE_missing}. They show that the MISE deteriorates as the number of missing values increases, but this tendency is not very strong.
\begin{figure}[hbt]
    \centering
    \includegraphics[width = 12cm]{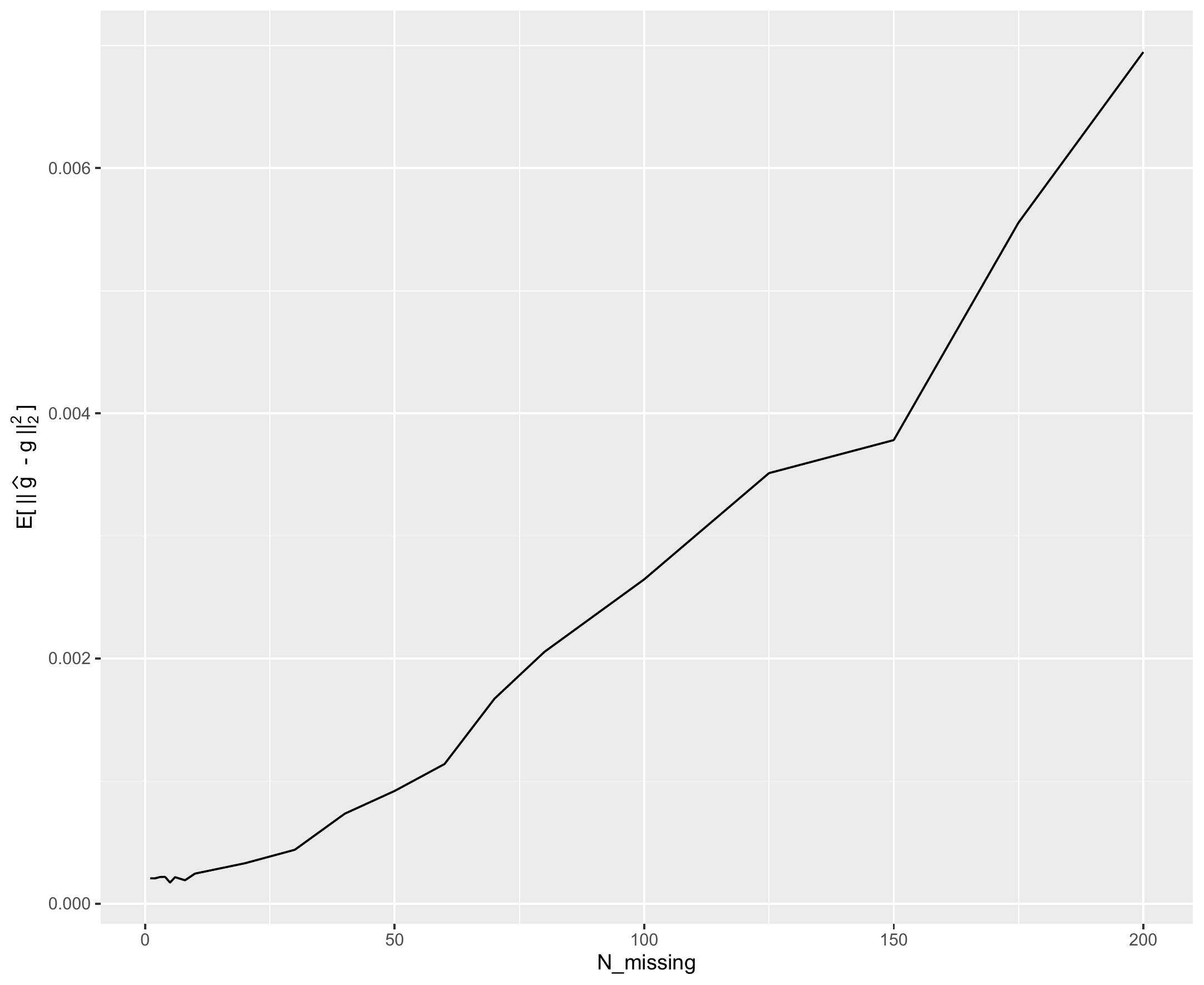}
   \caption{\textnormal{MISE of the estimate given by Algorithm~\ref{algo:basic_iteration_estimation_elliptical} as a function of the number of observations that contains missing values. The sample size is $n=1000$, the dimension is $d=3$ and the true generator is the Gaussian $g(x) = e^{-x}$ with the initialization ``identity''.}}
    \label{fig:MSE_missing}
\end{figure}


\section{Conclusion}
\label{conclusion}

We have stated some sufficient conditions to obtain the identifiability of the generator $g$ of a meta-elliptical copula. 
In many standard practical situations, they are satisfied, particularly when the corresponding correlation matrix $\Sigmabf$ is not the identity matrix.
Some inference procedures have been discussed, in particular an iterative method called MECIP that yields satisfying empirical results. 

\mds 

Among the avenues for further studies, a theoretically sound data-driven bandwidth selector would be welcome. 
\textcolor{black}{The theoretical properties of the iterative algorithm  (consistency, rate of convergence) remains unknown, even 
if the procedure seems to behave conveniently in our experiments. The proof of such results would be particularly challenging, due to the highly nonlinear analytical features of the maps we use in our algorithm MECIP.}
Moreover, it would be nice to weaken the conditions for the identifiability of $g$ when $\Sigmabf=\I_d$ (Proposition~\ref{identif_cop_ellip_Id}).
Finally, we conjecture that our results apply even when $g(0)=0$, a situation that was excluded for the sake of clarity in the theoretical developments, but that seems to be conveniently managed in our numerical experiments.

\bigskip

\noindent
{\bf Acknowledgements}
\noindent

\mds
Jean-David Fermanian has been supported by the labex Ecodec (reference project ANR-11-LABEX-0047).

\bigskip


\bibliographystyle{myjmva}
\bibliography{biblio}

\bigskip

\appendix

\section{A reminder about elliptical random vectors}
\label{marginals_elliptical}

Let $\X$ be a $d$-dimensional elliptical random vector, $\X\sim \Ec_d(\mubf,\Omegabf,g_d)$. When there is no ambiguity, $g_d$ will simply be denoted $g$. 
We recall a key representation of any elliptical random vector $\X$ in $\Rb^d$, that can be considered as a definition.

\begin{defi}[polar decomposition]
\textnormal{
\label{equiv_repres_ellipt}
    A random vector $\X$ follows an elliptical distribution $\Ec_d(\mubf,\Omegabf,g)$ if
    $\X \indistr{\mubf + R \A^\top \V}$, for some $d\times d$ matrix $\A$ such that $\A^\top \A=\Omegabf$, where $\V$ is a random vector uniformly distributed on the $d$-dimensional unit sphere and $R$ is
    a nonnegative r.v. independent of \textcolor{black}{$\V$}. When $R$ has a density $g_R$ with respect to the Lebesgue measure on $\Rb^+$ and no mass at zero, this is the case for $\X$ too.
    The density $f_{\X}$ of $\X$ is then given by 
        \begin{equation*}
    f_{\X}(\x) = {|\Omegabf|}^{-1/2} g\left( (\x-\mubf)^\top \, \Omegabf^{-1} \, (\x-\mubf) \right),\;\x \in \Rb^d,
    \end{equation*}
    and we have
    \begin{equation}
    g_R(r) = s_d r^{d-1} g(r^2), \; r\geq 0.
    \label{dens_R}
    \end{equation}
}
\end{defi}
See~\cite{gomez2003_survey}, Theorem 3, and~\cite{cambanis1981_theory}, Section 4.
We denote by $\X \sim \tilde\Ec_d(\mubf, \A, g_R)$ the latter ``polar decomposition''. It is said that $R$ is the modular variable of $\X$, and its density $g_R$ is the associated modular
density.
In our paper, we assume that $\X$ has a density $f_{\X}$ with respect to the Lebesgue measure on $\Rb^d$.
The support of $f_{\X}$ is $\{\x ; \sum_{j=1}^d x_j^2 < M\}$ for some $M\in \bar \Rb^+$ (\cite{kelker1970_distribution}, p.422).

\mds

The associated vector $\Y=\Omegabf^{-1/2}(\X-\mubf)$ is spherical, $\Y\sim\Ec_d(\0,\I_d,g_d)$.
It is well-known that sub-vectors of $\X$ (resp. $\Y$) are elliptical (resp. spherical) too: see Theorem 6 of Gomez et al. \cite{gomez2003_survey} (density generator),
 Embrechts et al. (p.10) \cite{embrechts2002_correlation} (characteristic generator), Theorem 1 of Kano \cite{kano1994_consistency}, among others.
\textcolor{black}{
See the recent surveys~\cite{babic2019comparison, paindaveine2014elliptical} about elliptical and meta-elliptical distributions.
}

\mds

Actually, our two parameterizations of elliptical distributions above are essentially unique.
\begin{prop} Let $\X$ be a $d$-dimensional random vector.
    \begin{itemize}
    \item[(a)]
    If $\X\sim \tilde\Ec_d(\mubf_1,\A_1,g_{R}^{(1)})$ and $\X\sim \tilde\Ec_d(\mubf_2,\A_2,g_{R}^{(2)})$, then there exists a constant $a>0$ such that
    $ \mubf_1=\mubf_2$, $ \A_1^\top \A_1= a \A_2^\top \A_2$, $ R_1 \indistr{a^{-1/2} R_2}$,    $g_{R}^{(1)}(t)= \sqrt{a}\, g_{R}^{(2)}(t\sqrt{a} ),$
    for almost every $t$.
    \item[(b)]
    If $\X\sim \Ec_d(\mubf_1,\Omegabf_1,g^{(1)})$ and $\X\sim \Ec_d(\mubf_2,\Omegabf_2,g^{(2)})$, then there exists a constant $a>0$ such that
    $ \mubf_1=\mubf_2$, $\Omegabf_1= a \Omegabf_2$, and $g^{(1)}(t)= a^{d/2} g^{(2)}(at),$
    for almost every $t$.
    \end{itemize}
    \label{prop:ellipt_distr_ident}
\end{prop}

\begin{proof}[\bf{Proof of Proposition~\ref{prop:ellipt_distr_ident}}]
Point $(a)$ is a consequence of Theorem 3 (ii) in \cite{cambanis1981_theory}. The proof of point $(b)$ is deduced from Proposition~\ref{equiv_repres_ellipt}: for $k\in \{1,2\}$, there exists
$(\A_k,g_{R^{(k)}})$ such that $\Ec_d(\mubf_k,\Omegabf_k,g^{(k)})\indistr \tilde \Ec_d(\mubf_k,\A_k,g_{R}^{(k)})$, and then apply point $(a)$. In particular, applying Eq.~(\ref{dens_R}) twice, we get
$$ s_d t^{(d-1)/2}g^{(1)}(t)=g_{R}^{(1)}(\sqrt{t})=\sqrt{a} g_{R}^{(2)}(\sqrt{at})
= s_d (at)^{(d-1)/2} \sqrt{a} g^{(2)}(at),$$
for every $t$, proving the asserted relationship between $g^{(1)}$ and $g^{(2)}$.
\end{proof}

\mds
As a consequence, in order to get unique sets of $\X$-parameters, it is necessary to impose an additional identification constraint.
Obviously, $\mubf=E[\X]$ is uniquely defined.
Moreover, due to Proposition~\ref{prop:ellipt_distr_ident}, $\Omegabf$ (resp. $\A$) is defined up to a positive constant.
Typically, when $\X$ has finite second moments, it is natural to impose $\text{Cov}(\X)=\Omegabf = \A^\top \A$.
Otherwise, it is still possible to impose $\text{Tr}(\Omegabf)=1$ or other similar constraints.
Once $\Omegabf$ is uniquely defined, this will be the case for $g$, as deduced from Proposition~\ref{prop:ellipt_distr_ident}.
In other words, the usual constraint
\begin{equation*}
    s_d\int_0^{\infty} r^{d-1} g(r^2)\, dr =s_d\int_0^{\infty} t^{d/2-1} g(t)\, dt/2 =1,
\end{equation*}
where $s_d:=2 \pi^{d/2}/\Gamma(d/2)$
plus a single additional constraint imply the identifiability of the law of any elliptically distributed random vector.

\mds

Let us specify how subvectors of an elliptically distributed vector are still elliptical, recalling Cambanis et al.~\cite{cambanis1981_theory} (Section 4) or~\cite{FangKotzNg} (Eq. (2.23)): if $\Y\sim \Ec_d(\0,\I_d,g_d)$, then the subvector $\Y_{(m)}$ of the first $m$
components of
$\Y_{(d)}$, $m< d$ is still spherical, i.e., $\Y_{(m)}\sim \Ec_m(\0,\I_m,g_m)$, where
\begin{align}
    g_m(u)= s_{n-m}\int_0^{\infty} g_d(u+r^2) r^{d-m-1} \, dr,
    \label{eq:gen_subvector}
\end{align}
where $s_k$ denotes the surface of the $k$-dimensional unit sphere in $\Rb^d$, i.e., $s_k= 2\pi^{k/2}/\Gamma (k/2)$, $k\geq 1$.
We deduce
\begin{equation}
 g_1(u)= \frac{\pi^{(d-1)/2}}{\Gamma((d-1)/2)}\int_0^{\infty} g_d(u+s) s^{(d-3)/2} \, ds.
 \label{margin_ellip}
\end{equation}
Note that $X_{k}\sim \Ec(0,\Omega_{kk},g_{1})$ for any $k\in \{1,\ldots,d\}$, where $\Omegabf=:[\Omega_{k,l}]$ and $g_1$ as above (\cite{gomez2003_survey}, Th. 6).
Therefore, its density is $f_k(t)=g_1(t^2/\Omega_{kk})/\sqrt{\Omega_{kk}}$ for every real $t$.
When $\Omega_{kk}=1$, as in the case of correlation matrices $\Omegabf$, the density of any margin of $\Ec_d(\0,\Omegabf,g)$ is
\begin{equation}
f_g(t):=g_1(t^2)=\frac{s_{d-1}}{2}\int_0^{\infty} g(t^2+s) s^{(d-3)/2} \, ds=
s_{d-1}\int_{0}^{\infty} g(t^2+r^2) r^{d-2} \, dr.
\label{margin_ell_reduced}
\end{equation}
Clearly, $f_g$ is even. Denote by $a_g^2 \in \bar \Rb^+$ the upper bound of $g$'s support (possibly equal to $+\infty$).
Then, it is easy to see that the support of $f_g$ is $(- a_g, a_g)$, possibly including its boundaries.
When $g$ is bounded in a neighborhood of a positive real number $t$, then $g_1$ is continuous at \textcolor{black}{$t$} too (\cite{kelker1970_distribution}, Lemma 3).
As a consequence, $f_g$ is continuous at $\pm \sqrt{t}$.
Finally, Theorem 6 in Gomez et al. \cite{gomez2003_survey} provides similar relationships in terms of modular variables.



\mds

Let us discuss the nonparametric inference of density generators.
Assume we observe an i.i.d. sample $(\X_1,\ldots,\X_n)$ drawn along $\X\sim \Ec_d(\mubf,\Sigmabf,g)$ for some correlation matrix $\Sigmabf$.
The statistical estimation of the underlying parameters has been studied in the literature.
First, it is easy to evaluate $\mubf$ by $\hat\mubf$, the empirical mean of the vectors $\X_i$, $i\in \{1,\ldots,n\}$. 
Second, an estimator $\hat\Sigmabf$ of $\Sigmabf$ can be obtained by empirical Kendall's taus'.
Now, set the standardized vector $\Y := \Sigmabf^{-1/2} \, (\X-\mubf)$, and the standardized observations $\hat\Y_i := \hat\Sigmabf^{-1/2} \, 
(\X_i-\hat\mubf)$, $i\in \{1,\ldots,n\}$.
Note that the latter observations are identically distributed but not independent.

\mds

Third, we know that the law of $\Y$ is spherical, $\Y\sim \Ec_d(\0,\I_d,g)$, and depends on $g$ only.
The task of estimating $g$ is relatively challenging and a few techniques have been proposed in the literature.
In  Battey and Linton~\cite{battey2014nonparametric}, the estimation of the density generator is done by a finite mixture sieve.
Bhattacharyya \cite{bhattacharyya2013_study} proposed piecewise constant  estimators of $g$, that are fitted by log-linear splines.
In \cite{BhattacharyyaBickeL_adaptive}, they also provided an EM-algorithm to estimate the parameters of elliptical distributions mixtures.
Stute and Werner \cite{stute1991_nonparametric} introduced a usual kernel density estimator of $\| \Y \|^2$'s law. Since $\|\Y \|^2=R^2$ through the polar decomposition of $\X$, and
invoking~(\ref{dens_R}), they deduced the following estimator of the $\Y$ density:
\begin{equation}
f_n(\y) = \frac{2}{s_d \, n \, h_n \, {||\y||}^{d-2}} \sum_{i=1}^{n}
K\left( \frac{{||\y||}^2 - {||\hat\Y_i||}^2}{h_n}   \right), \; \y\in \Rb^d,
\end{equation}
{\color{black}for a bandwidth $h_n > 0$ and a one-dimensional kernel $K(\cdot)$.}
Since $f_Y(\y)=g(\| \y\|^2)$, an estimate of the density generator $g$ itself is 
\begin{equation}
\hat g(u)  = \frac{2}{s_d n \, h_n \, u^{(d-2)/2}} \sum_{i=1}^{n}
K\left( \frac{u - {||\hat\Y_i||}^2}{h_n}   \right), \; u\in \Rb^+.
\label{estim_g_WernerStute}
\end{equation}

\mds

These ideas have been refined by
Liebscher \cite{liebscher2005_semiparametric} and applied by \cite{pimenova2012_semiparametric}, notably. He noticed that the estimator~(\ref{estim_g_WernerStute}) does not behave
conveniently close to the origin.
He proposed to nonlinearly transform the data before invoking a kernel estimator, to avoid convergence problems in boundary regions.

\section{Finite distance properties of \texorpdfstring{$\hat\theta_{n,m}$}{\^theta\_\{n,m\}}}
\label{prop_thetahat_finite_dist}

To this end, introduce the pseudo-true value of the parameter
$\theta_m^*$. It is the ``best'' value of the parameter when the model is assumed to belong to $\Gc_m$: generally speaking, for a divergence 
$D(\cdot,\cdot)$ between distributions on $\Rb^d$, define
\begin{equation}
\theta_m^*:= \underset{\theta: \, g_\theta \in \Gc_m}{\arg \min} \, D\big(\Pb_0,\Pb_{g_\theta,\Sigmabf,F_1,\ldots,F_d}\big),
\end{equation}
where $\Pb_0$ is the law of the observations (the true DGP) and $\Pb_{\bar g,\bar\Sigmabf,\bar F_1,\ldots,\bar F_d}$ is the law induced by 
a Trans-elliptical distribution $\Tc\Ec_d( \bar g, \bar\Sigmabf,\bar F_1,\ldots,\bar F_d)$.
Here, we set
$$ D\big(\Pb_0,\Pb_{g_\theta,\Sigmabf,F_1,\ldots,F_d}\big) := E\big[\Gb_n(\theta,\Uc)\big],$$
and we assume $\theta_m^*$ satisfies the first-order condition $E\big[\nabla_\theta\Gb_n(\theta_m^*,\Uc)\big]=0$.

\mds

\begin{cond}
\textnormal{
\label{H1}
The pseudo-true parameters are ``sparse'': $card(\Ac_m) = k_m < p_m$ and $\Ac_m = \{i: \theta^*_{m,i} \neq 0\}$.
}
\end{cond}

\begin{cond}
\textnormal{
\label{amenable}
We consider coordinate-separable penalty (or regularizer) functions  $\pp_n: \Rb_+ \times \Rb^{p_n} \rightarrow \Rb$, i.e.,
\textcolor{black}{$\pp_n(\lambda_n,\theta) = \overset{p_n}{\underset{k = 1}{\sum}} p(\lambda_n,\theta_k)$}.
Moreover, for some $\mu \geq 0$, the regul\textcolor{black}{ar}izer $\pp_n(\lambda_n,\cdot)$ is assumed to be $\mu$-amenable, in the sense that
\begin{itemize}
    \item[(i)] $\rho \mapsto p(\lambda_n,\rho)$ is symmetric around zero and $p_n(\lambda_n,0) = 0$.
    \item[(ii)] $\rho \mapsto p(\lambda_n,\rho)$ is non-decreasing on $\Rb_+$.
    \item[(iii)] $\rho \mapsto p(\lambda_n,\rho) / \rho$ is non-increasing on $\Rb_+\setminus \{0\}$.
    \item[(iv)] $\rho \mapsto p(\lambda_n,\rho)$ is differentiable for any $\rho \neq 0$.
    \item[(v)] $\underset{\rho \rightarrow 0^+}{\lim} p'(\lambda_n,\rho) = \lambda_n$.
    \item[(vi)] $\rho \mapsto p(\lambda_n,\rho) + \mu \rho^2/2$ is convex for some $\mu \geq 0$. \\
    \mds
    The regularizer $\pp_n(\lambda_n,.)$ is said $(\mu,\gamma)$-amenable if, in addition,
    \item[(vii)] There exists $\gamma \in (0,\infty)$ such that $p'(\lambda_n,\rho) = 0$ for $\rho \geq \lambda_n \gamma$.
\end{itemize}
}
\end{cond}

\mds

The latter assumption provides regularity conditions to potentially manage non-convex penalty functions. These regularity conditions are the same as in~\cite{loh2017statistical, loh2017support, loh2015regularized}.
Some usual penalties are the Lasso, the SCAD (\cite{fan2001variable}) and the MCP (\cite{zhang2010nearly}),
given by
\begin{equation*}
\begin{array}{llll}
\mathbf{Lasso:} \;\; p(\lambda_n,\rho)  & = & \lambda_n |\rho|, \\
\mathbf{MCP:} \;\; p(\lambda_n,\rho)  & = & sign(\rho) \lambda_n \int^{|\rho|}_0 \{1-z/(\lambda_n b_{mcp})\}_+ dz, \\
\mathbf{SCAD:} \;\; p(\lambda_n,\rho)  & = & \begin{cases}
\lambda_n |\rho|, & \text{for} \; \; |\rho| \leq \lambda_n, \\
-(\rho^2-2b_{scad}\lambda_n|\rho|+\lambda^2_n)/\{2(b_{scad}-1)\}, & \text{for} \;\; \lambda_n \leq |\rho| \leq b_{scad} \lambda_n, \\
(b_{scad}+1)\lambda^2_n/2, & \text{for} \;\; |\rho| > b_{scad} \lambda_n,
\end{cases}
\end{array}
\end{equation*}
where $b_{scad}>2$ and $b_{mcp}>0$ are fixed parameters for the SCAD and MCP respectively. The Lasso is a $\mu$-amenable regularizer, whereas the SCAD  and the MCP regularizers are
$(\mu,\gamma)$-amenable. More precisely, $\mu$ is equal to zero, $1/(b_{scad}-1)$ or $1/b_{mcp}$ for the Lasso, SCAD or MCP respectively.

\mds

As for many parametric models, numerous empirical log-likelihoods associated to copulas are not concave
functions in their parameters, at finite distance and globally on $\Theta_m$. Moreover,
this is still the case for some popular regularizers, as SCAD.
The restricted strong convexity is a key ingredient that allows the management of non-convex loss functions. Intuitively, we would like to handle a loss function that locally admits some curvature around its optimum. The latter one can be due to the discrepancy between the empirical loss and the ``true'' loss, to a non-convex penalty or even to the use of pseudo-observations instead of usual ones.

\mds

We say that an empirical loss function $\Lb_n$ satisfies the restricted strong convexity condition (RSC) at $\theta$ if there exist two positive functions
$ \alpha_1,\alpha_2$ and two nonnegative functions \textcolor{black}{$\nu_1,\nu_2 $} of $(\theta,n,d)$ such that, for any $\Delta \in \Rb^d$,
\textcolor{black}{\begin{eqnarray*}
\langle\nabla_{\theta} \Lb_n(\theta + \Delta) - \nabla_{\theta} \Lb_n(\theta),\Delta \rangle  & \geq & \alpha_1 \|\Delta\|^2_2 - \nu_1  \|\Delta\|^2_1,\; \text{if } \|\Delta\|_2 \leq 1,
\label{RSC_1} \\
\langle\nabla_{\theta} \Lb_n(\theta + \Delta) - \nabla_{\theta} \Lb_n(\theta),\Delta \rangle & \geq & \alpha_2 \|\Delta\|_2 - \nu_2  \|\Delta\|_1,\; \text{if } \|\Delta\|_2 > 1.
\label{RSC_2}
\end{eqnarray*}}
Note that the (RSC) property is fundamentally local and that $\alpha_k,\nu_k$, $k\in \{1,2\}$ depend on the chosen $\theta$.
To weaken notations, we simply write $\alpha_k$ and $\nu_k$, $k\in \{1,2\}$, by skipping their implicit arguments $(\theta,n,d)$.

\begin{thm}[Poignard and Fermanian, 2021]\label{bound_param}
Suppose the objective function $\Gb_n(\cdot,\widehat\Uc): \Theta_m\subset\Rb^{p_m} \rightarrow \Rb$ satisfies the (RSC) condition at $\theta_m^*$. Moreover, $\pp(\lambda_n,\cdot)$
is assumed to be $\mu$-amenable, with $3\mu < 4\alpha_1$ and $4R\nu_2 \leq \alpha_2$. Assume
\begin{equation} \label{tuning}
4 \max\Big\{ \|\nabla_{\theta} \Gb_n(\theta_m^*,\widehat\Uc)\|_{\infty},2R\nu_1 \Big\} \leq \lambda_n \leq \cfrac{\alpha_2}{6R}\cdot
\end{equation}
Then, for every $n$, any stationary point $\hat{\theta}_{n,m}$ of (\ref{Mestimator_theta}) satisfies
\begin{equation*}
    \|\hat{\theta}_{n,m} - \theta_m^*\|_2 \leq \cfrac{6 \lambda_n \sqrt{k_m} }{4 \alpha_1 - 3 \mu}, \; \text{and}\;\; \|\hat{\theta}_{n,m} - \theta_m^*\|_1 \leq
\cfrac{6 (16 \alpha_1 - 9\mu)}{(4\alpha_1 - 3\mu )^2 } \lambda_n k_m. 
\end{equation*}
\end{thm}
The proof straightforward extends Theorem 1 in~\cite{poignard2021finite}.
In particular, if the ``true'' generator $g$ of the data belongs to some subset $\Gc_{m_0}$ (the model is correctly specified from
 the index $m_0$ on), then $\theta_{m_0}^*=\theta_0$.
When we work with true realizations of $\U$, a typical behavior is $\lambda_n \asymp (\ln p_n/n)^{1/2}$ (see~\cite{loh2017statistical}). As discussed in~\cite{poignard2021finite}, this rate is getting worse in general when dealing with pseudo-observations in $\widehat\Uc$: $\lambda_n$ will be at most of order $d p_n^2 /\sqrt{n}$.

\end{document}

\begin{figure}[htb]
    \centering
    \begin{tikzpicture}[node distance = 1.5cm]
    
    \node[block] (dataset) {Dataset $(X_1, \dots, X_n)$};
    
    \node[below of=dataset] (belowDS) {};
    \node[block, below of=dataset, node distance = 2cm] (pseudoobs) {Pseudo-observations on $[0,1]^d$ \\
    $\hat U_{i,j} = \hat F_j(X_{i,j})$};
    \path[draw, arr] (dataset) -- (pseudoobs) node [midway](belowDS){} ;
    
    \node[block, right of=belowDS, node distance = 4cm] (empcdf) {Marginal empirical cdfs \\ $\hat F_1, \dots, \hat F_d$};
    \path[draw, arr] (empcdf) -| (pseudoobs);
    \path[draw, arr] (dataset) -| (empcdf);
    
    \node[block, left of=belowDS, node distance = 4.2cm] (corrmat) {Estimated correlation \\ matrix $\hat \Sigmabf$};
    \path[draw, arr] (dataset) -| (corrmat);
    
    \node[block, below of=pseudoobs, node distance = 6.5cm] (pseudoobsRd) {Pseudo-observations on $\Rb^d$ \\ $\hat Z_{i,j} = Q_{\hat g}(\hat U_{i,j})$};
    \path[draw, arr] (pseudoobs) -- (pseudoobsRd) node [pos=0.82](belowPseudoObs){} ;
    
    \node[block, right of=belowPseudoObs, node distance = 5cm] (univQuantiles) {Estimated univariate quantile $Q_{\hat g}$ \\
    computed using the current $\hat g^{(N)}$};
    \path[draw, arr_iter] (univQuantiles) -| (pseudoobsRd) ;
    
    \node[block, below of=pseudoobsRd, node distance = 2cm] (pseudoobsRdcomp) {Imputation of the missing pseudo-observations \\ to have complete vectors $\hat \Z_1, \dots, \hat \Z_n$ using $\hat g$};
    \path[draw, arr_iter] (pseudoobsRd) -- (pseudoobsRdcomp) ;
    \path[draw,->] (corrmat) |- (pseudoobsRdcomp);
    
    \node[block, below of=pseudoobsRdcomp, node distance = 2cm] (transSample) {Transformed sample \\
    $Y_i = \psi_a \big(\hat \Z_{i}^\top \, \hat \Sigmabf \, \hat \Z_{i} \big)$
    };
    \path[draw, arr_iter] (pseudoobsRdcomp) -- (transSample) ;
    \path[draw,->] (corrmat) |- (transSample);
    
    \node[block, below of=transSample, node distance = 2cm] (esti_g) {Estimated generator using the one-dimensional sample $(Y_1, \dots, Y_n)$ \\
    $\tilde g(t) = s_d^{-1} \psi_a'(t) t^{-d/2+1} (n h)^{-1} \sum_{i=1}^n \bigg(
    K \Big( \big(\psi_a(t) - Y_i\big)/h\Big)
    + K \Big( \big(\psi_a(t) + Y_i\big)/h\Big)
    \bigg)$
    };
    \path[draw, arr_iter] (transSample) -- (esti_g) ;
    
    \node[block, below of=esti_g, node distance = 2cm] (normalized_g) {Normalized generator \\
    $\hat g = \texttt{Algorithm \ref{algo:normalization_g}} (\tilde g)$
    };
    \path[draw, arr_iter] (esti_g) -- (normalized_g) ;
    \path[draw, arr_iter] (normalized_g) -| (univQuantiles.342) node [pos=0.25, below, align=center]{if not converged, $N \leftarrow N+1$}; 
    
    \node[block, below of=normalized_g, node distance = 2cm] (final_g) {Final estimator $\hat g^{(\infty)}$};
    \path[draw, arr_startEnd] (normalized_g) -- (final_g) node [pos=0.5, right]{if converged}; 
    
    \node[block, above of=univQuantiles, node distance = 3.7cm] (init_g) {Initial value $\tilde g^{(0)}$};
    
    \node[block, below of=init_g] (initNormed_g) {Normalized initial value \\
    $\hat g^{(0)} = \texttt{Algorithm \ref{algo:normalization_g}} (\tilde g^{(0)})$
    };
    \path[draw,arr] (init_g) -- (initNormed_g) ;
    \path[draw,arr_startEnd] (initNormed_g) -- (univQuantiles) node [pos=0.5, right]{Start of the iterations, $N \leftarrow 0$};
    \path[draw,arr] (pseudoobs.340) |- (init_g) ;

    \node (initbox) [draw, rounded corners, dashed, inner sep=1ex, 
        fit=(dataset) (corrmat) (init_g) (initNormed_g),
        label={[node font=\small\itshape\bfseries,rotate=90,anchor=south]left: Initialization}
        ]
        {};
                
    \node (lieb) [draw, rounded corners, densely dotted, black, inner sep=1ex, 
        fit=(transSample) (esti_g),
        label={[node font=\small\itshape\bfseries,rotate=90,anchor=south, name=labelLieb]left: Liebscher's procedure}
        ]
        {};
                
    \node (iteration) [draw, blue, rounded corners, dashed, inner sep=1ex, 
        fit=(lieb) (normalized_g) (univQuantiles) (labelLieb),
        label={[node font=\small\itshape\bfseries,rotate=90,anchor=south,color=blue]left: Iterative procedure}
        ]
        {};
    
    \end{tikzpicture}
    
    \caption{\textnormal{Flowchart of the iterative estimation procedure MECIP, with missing values. The thick blue loop denotes the iterative procedure.}}
    \label{fig:my_label}
\end{figure}